\documentclass[review,onefignum,onetabnum]{siamart190516}

\usepackage[textsize=scriptsize,textwidth=25mm]{todonotes}
\usepackage{a4wide}

\usepackage{lipsum}
\usepackage{amsfonts}
\usepackage{graphicx}
\usepackage{epstopdf}
\usepackage{amsmath}
\allowdisplaybreaks
\usepackage{algorithmic}
\DeclareMathOperator\supp{supp}
\usepackage{caption,stackengine,graphicx,subcaption,subfloat}
\stackMath
\newcommand\frightarrow{\scalebox{1}[.4]{$\rightarrow$}}
\newcommand\darrow[1][]{\mathrel{\stackon[1pt]{\stackanchor[1pt]{\frightarrow}{\frightarrow}}{\scriptstyle#1}}}
\usepackage{xparse}

\NewDocumentCommand{\tens}{t_}
{%
\IfBooleanTF{#1}
{\tensop}
{\otimes}%
}
\NewDocumentCommand{\tensop}{m}
{%
\mathbin{\mathop{\otimes}\displaylimits_{#1}}%
}
\usepackage{tikz,pgfplots} 
\usetikzlibrary{decorations.markings,calc,fit,intersections} 
\usepackage{mathtools} 
\usepackage{amssymb} 
 
\ifpdf
\DeclareGraphicsExtensions{.eps,.pdf,.png,.jpg}
\else
\DeclareGraphicsExtensions{.eps}
\fi

\usepackage{amsmath}

\newsiamremark{remark}{Remark}
\newsiamremark{hypothesis}{Hypothesis}
\crefname{hypothesis}{Hypothesis}{Hypotheses}
\newsiamthm{claim}{Claim}
\newsiamthm{assumption}{Assumption}
\newsiamthm{example}{Example}
\usepackage{lipsum}
\usepackage{amsopn}

\newcommand{\revision}[1]{\textcolor{black}{#1}}

\headers{Hybrid models for multilane-multiclass traffic}{X. Gong, B. Piccoli, and G. Visconti}

\title{Mean-field limit of a hybrid system for multi-lane multi-class traffic\thanks{Submitted on \today.
\funding{The research of X.~G. was partially supported by the NSF CPS Synergy project "Smoothing Traffic via Energy-efficient Autonomous Driving" (STEAD) CNS 1837481.\newline
The research of B.~P. is based upon work supported by the U.S. Department of Energy’s Office of Energy Efficiency and Renewable Energy (EERE) under the Vehicle Technologies Office award number CID DE-EE0008872. The views expressed herein do not necessarily represent the views of the U.S. Department of Energy or the United States Government.\newline
The research of G.~V. has been funded by the Deutsche Forschungsgemeinschaft (DFG, German Research Foundation) under Germany’s Excellence Strategy – EXC-2023 Internet of Production – 390621612.
}}}

\author{Xiaoqian Gong\thanks{School of Mathematics and Statistical Sciences, Arizona State Univeristy, Tempe, AZ, USA (\email{xgong14@asu.edu})}
\and Benedetto Piccoli\thanks{Department of Mathematical Sciences and Center for Computational and Integrative Biology, Rutgers University, Camden, NJ, USA (\email{piccoli@camden.rutgers.edu})}
\and Giuseppe Visconti\thanks{Department of Mathematics, Sapienza University of Rome, Italy (\email{giuseppe.visconti@uniroma1.it}). This work started at the Institut f\"{u}r Geometrie und Praktische Mathematik, RWTH Aachen University, Aachen, Germany.}
}

\begin{document}

\maketitle

\begin{abstract}
This article aims to study coupled mean-field equation and ODEs with discrete events motivated by vehicular traffic flow. \revision{Precisely,} multi-lane traffic flow in presence of human-driven and autonomous vehicles is considered, with the autonomous vehicles possibly influenced by external policy makers. First a finite-dimensional hybrid system is developed based on the continuous Bando-Follow-the-Leader dynamics coupled with discrete events due to \revision{lane-change maneuvers}. 
\revision{Then the} mean-field limit of the finite-dimensional hybrid system is rigorously derived \revision{for the dynamics of the} human-driven vehicles\revision{. The microscopic lane-change maneuvers of the human-driven vehicles generates a source term to the mean-field PDE. This leads to} an infinite-dimensional hybrid system\revision{, which is} described by coupled Vlasov-type PDE, ODEs and discrete events. 
\end{abstract}



\section{Introduction}
Mathematical traffic models, depending on the scale at which they represent vehicular traffic, usually can be classified into different categories: microscopic, mesoscopic, macroscopic, and cellular. We refer to the survey papers~\cite{albi2019vehicular, bellomo2011modeling,piccoli2009vehicular}, and reference therein, for general discussions about the models at various scales in the literature. In this paper, we focus on microscopic models and mesoscopic descriptions. 

Microscopic models are discrete models of traffic flow that study the behavior of individual vehicles and predict their trajectories by means of ordinary differential equations (ODEs). One such model is the combined Bando~\cite{bando1994structure} and Follow-the-Leader~\cite{FTL1961, reuschel1950vehicle1, reuschel1950vehicle} model that concerns both relaxation to an optimal velocity and interactions with the closest neighboring vehicle ahead. Mean-field equations, and in general models based on partial differential equations (PDEs), treat vehicular traffic as fluid flow, and aim to provide an aggregate and statistical viewpoint of traffic by capturing and predicting the main phenomenology of the microscopic dynamics. Within this context we would like to mention the most classical works~\cite{paveri1975TR,Prigogine61,PrigogineHerman} and recent developments, e.g.~\cite{coscia2007IJNM,DelitalaTosin2007,HertyPareschi2010,klar1997Enskog,PiccoliTosinZanella}. This scale of representation is \revision{therefore} useful and accurate in the limit of the dynamical system with \revision{a large number of} vehicles. \revision{The link} between the \revision{microscopic and the mesoscopic} description can be also rigorously established in generalized Wasserstain distance~\cite{Golse}. We point-out that this discussion is not restricted to traffic flow and is of interest in many research areas, such as in biology~\cite{CouzinKrauseFranksLevin2005,CuckerSmale2007a} or social~\cite{CristianiPiccoliTosin2011} and economic dynamics~\cite{TrimbornPareschiFrank2019}.

In the present work, we aim to develop and study qualitative properties of models for traffic which are motivated by the idea of considering, simultaneously, two important aspects: lane-change maneuvers and heterogeneous composition of the flow. The former is one of the most common maneuvers\revision{, source of interaction and risk~\cite{HertyVisconti2018},} among vehicles on motorways. Currently, multi-lane traffic is modeled either by two-dimensional models~\cite{HertyMoutariVisconti2018,SukhinovaTrapeznikovaChetverushkinChurbanova2009}, in which lane changing rules are not explicitly prescribed, or by treating lanes as discrete entities~\cite{HoldenRisebro2019,SongKarni2019}. The latter aspect, instead, is becoming more and more important with the increasingly interest in automated-driven vehicles and their effects within the vehicular traffic flow~\cite{HoogendoormReview2014}. \revision{Experiments~\cite{Piccoli-DissipationStopAndGo2018,DelleMonache2019} and mathematical models~\cite{Piccoli2020,TosinZanella2021}
have shown that a small number of controlled vehicles can stabilize traffic flow damping unstable phenomena.}

\revision{The main contributions of this paper are described in the following.} We \revision{define} microscopic dynamics \revision{for two} classes of vehicles, one identified by human-driven vehicles and the other one by automated-driven vehicles. We use a Bando-Follow-the-Leader model for both classes. \revision{More precisely, the model} is reformulated by replacing the interaction with the closest vehicle ahead by a short-range interaction kernel which allows to write the system of ODEs in a convolution framework. \revision{Furthermore,} the dynamics of autonomous vehicles differs from the dynamics of human-driven ones due to an additive control term which, in applications, may be provided by a remote controller~\cite{Piccoli-DissipationStopAndGo2018}. Along with the continuous dynamics, we consider discrete dynamics generated by the lane changing rules, which are
designed following \cite{KTH07}. 
The presence of both continuous and discrete dynamics leads us to a hybrid system, see~\cite{654885,garavello2005hybrid,4806347,piccoli1998hybrid,Tomlin_1998}. Finally, we perform a mean-field limit for human-driven vehicles only, since autonomous vehicles are supposed to be a small percentage of the total flow on motorways. The trajectories of the hybrid system
exhibit dicontinuities thus the limit procedure requires a generalization
of the classical Arzelà-Ascoli Theorem. 
This leads to a Vlasov-type PDEs with a source term~\cite{festa2018mean,HertyIllnerKlarPanferov,IllnerKlarMaterne}, which is generated by the discrete lane changing rules. Such source term induces
the measure solutions to change mass in time, thus the limit is obtained
using the generalized Wasserstein distance \cite{piccoli2014generalized}.
Together with the continuous and discrete dynamics of the autonomous vehicles, we obtain a hybrid system with mean-field limit involved,
for which we prove existence and uniqueness of solutions.

Our main result is thus a complete representation of multi-lane
multi-class hybrid system at microscopic and mesoscopic scales
together connected by a rigorous limiting procedure.
\revision{Namely, we prove the convergence of the finite dimensional hybrid system to the corresponding infinite dimensional hybrid system in Theorem \ref{thm_main}.}
This framework allows to study optimal control problems at multiple scales,
in the same spirit as \cite{BFR17,FS14}. \revision{The optimal control problems associated with the finite and infinite dimensional hybrid systems were investigated in~\cite{GongPiccoliVisconti2021}.}
We also notice that, even if our main example is vehicular traffic,
the same framework may be adapted to model any hybrid system
with multi-population at microscopic and mesoscopic scale,
including social and crowd dynamics~\cite{fornasier2014mean}.

The paper is organized as follows. In Section \ref{sec_notation}, we briefly recall the basic models, notions, notations and preliminaries used in this article. Section \ref{sec_finite_dimen_hybrid} devotes to the definition of lane changing conditions and the study of well-posedness of the finite-dimensional hybrid system modeling multi-lane traffic at the microscopic level. In Section \ref{sec_mean_field_limit_dimen_hybrid}, we define a hybrid system involving mean-field limit of the finite-dimensional hybrid system involving human-driven vehicles and prove the existence and uniqueness of the trajectories of the mean-field hybrid system. Finally, Section~\ref{sec_conclusions} ends the paper with conclusions and outlook. 

\section{Notations, Definitions and Preliminaries}
\label{sec_notation}
In this section, we first recall some basic notations and definitions about traffic flow models\revision{, hybrid systems} and the generalized Wasserstein distance we use in this article. Then we list some well-known results about solutions to Carath\'eodory differential equations and to partial differential equations of Vlasov-type with source term. At last, we give a proof to a revised version of Arzelà-Ascoli Theorem. 

\subsection{Traffic Flow Models} \label{sec_models}
In order to setup the mathematical formulation, in the following we consider a population of $P$ cars on an open stretch road. 
To each vehicle, labeled by an index $i\in \{1, \dots, P\}$, we associate a vector of indices $\iota(i) = (i, i_L, i_F)$. Here 
$i_L$ is the index of \revision{the leader, i.e.~}the vehicle in front of vehicle $i$, 
and $i_F$ is the index of \revision{the follower, i.e.~}the vehicle flowing vehicle $i$.
\revision{
Typically, the labels are assigned based on the increasing position of vehicles on the road, in such a way the first vehicle is labeled as $1$, the second as $2$, and so on. In this way, we have that, for each $i\in\{1,\dots,P\}$, $i_L=i+1$ and $i_F=i-1$. However, for the purpose of this paper, we avoid the introduction of any ordering. Labels can be randomly assigned among vehicles on the road and they remain unchanged. Then, the indices $i_L$ and $i_F$ are the labels of the closest vehicle ahead and behind the reference vehicle $i$, respectively. To fix notation, we let $(x_i, v_i)$ be the vector of position-velocity, with $x_i \in \mathbb{R}, v_i \geq 0$, of vehicle $i$. Then,
\begin{equation} \label{L_F_labels}
        i_L = \underset{\substack{k\in\{1,\dots,P\}\\x_k>x_i}}{\arg\min} x_k-x_i, \quad
        i_F = \underset{\substack{k\in\{1,\dots,P\}\\x_k<x_i}}{\arg\min} x_i-x_k.
\end{equation}
In addition, we assign $i_L=0$ if vehicle $i$ is the \revision{last} on the road, that is, $i=\underset{k\in \{1, \dots, P\}}{\arg\max} x_k$. Similarly, $i_F =0$ if vehicle $i$ is the \revision{first} on the road, that is,  $i=\underset{k\in \{1, \dots, P\}}{\arg\min} x_k$.
\begin{figure}
  \centering
    \includegraphics[width=0.49\textwidth]{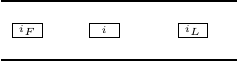}
    \caption{Schematic representation of the vehicle labeling in the single lane model.\label{fig:labels}}
\end{figure}
In Figure~\ref{fig:labels} we schematically summarize the notation.}

The Follow-the-Leader (FtL) model, which was introduced in \cite{reuschel1950vehicle1, reuschel1950vehicle}, assumes that the acceleration of a vehicle is directly proportional to difference between the velocity of the vehicle in front and its own velocity, and is inversely proportional to their distance. \revision{Let $h_i = x_{i_L} - x_i$ be the headway of the $i$-th vehicle.} The main dynamics described by the FtL model is given by 
\begin{equation}
\label{FtL_eqns}
\left\{
\begin{array}{ll}
\dot{x}_i = v_i,\\
\dot{v}_i = \beta_i \frac{v_{i_L} - v_i}{(h_i)^2}, \quad i \in \{1, \dots, P\}
,
\end{array} 
\right. 
\end{equation}
where $\beta_i$ is a positive parameter with appropriate dimension. If vehicle $i$ is the \revision{last} vehicle, 
then the dynamics of vehicle $i$ is given by 
$\dot{x}_i = v_{\max}$,
where $v_{\max}$ is a given maximum velocity, perhaps the speed limit. By system \eqref{FtL_eqns}, one can see also a drawback of the FtL model: as long as the relative velocity $\Delta v_i = v_{i_L}-v_i$ is zero, the acceleration is zero. That is to say, even at high speeds, an extremely small headway is allowed.

The Bando model, proposed by Bando et al. in \cite{bando1994structure}, fixed the aforementioned problems by associating each vehicle an optimal velocity function $V$ which describes the desired velocity for the headway. A driver controls the acceleration or deceleration based on the difference between his/her own velocity and the optimal velocity. The optimal velocity is typically an increasing function of the headway, namely it tends to zero for small headways and to the maximum value $v_{\max}$ for large headways. 
The governing equation of the Bando model is as follows: 
\begin{equation}
\label{eqn_OV_model}
\left\{
\begin{array}{ll}
\dot{x}_i = v_i,\\
\dot{v}_i = \alpha_i(V(h_i) - v_i), \quad i\in \{1, \dots, P\}, 
\end{array}
\right.
\end{equation}
where $\alpha_i$ is a positive parameter denoting the speed of response. The equilibrium point for this model is obtained when all vehicles travel at constant speed and have the same headway, see \cite{konishi2000decentralized}.

For the combined Bando-FtL model, which represents the basic model we consider in this work, the dynamics of the $i$-th vehicle is defined as follows: 
If $i_L \not = 0$, i.e., if vehicle $i$ is not the \revision{last}, 
then 
\begin{equation}
\label{eqn_bando_FtL}
\left\{
\begin{array}{ll}
\dot{x}_i = v_i,\\
\dot{v}_i = \alpha_i(V(h_i) - v_i) + \beta_i \frac{v_{i_L} - v_i}{(h_i)^2}, \quad i \in \{1, \dots, P\},
\end{array}
\right.
\end{equation}
where the headway is $h_i = x_{i_L} -x_i$. For simplicity, we take $\alpha_i = \alpha$, $\beta_i =\beta$ for all $i \in \{1, \dots, P\}$. 

\subsubsection{\revision{Convolution form of the Bando-FtL model}}
Now we will rewrite the Bando-FtL model, system \eqref{eqn_bando_FtL}, in convolution form to justify the fact that drivers adjust their acceleration or deceleration according to the velocities of their front nearby vehicles, their own velocities and optimal velocities. For $T>0$ fixed and $i=1, \dots, P$, define a time dependent atomic probability measure on $\mathbb{R}\times \mathbb{R}^{+}_0$,
\begin{equation}
\label{eqn_atomic_measure_P}
\mu_{P} (t)= \frac{1}{P} \sum \limits_{i=1}^{P} \delta_{\left(x_i(t), v_i(t)\right)}
\end{equation}
supported on an absolutely continuous trajectories $t \in [0, T] \mapsto (x_i(t), v_i(t)) \in \mathbb{R}\times \mathbb{R}^{+}_0$. 
Let $\varepsilon_0 >0$ be fixed.
Define a convolution kernel $H_1 \colon \mathbb{R} \times \mathbb{R}^{+}_0 \revision{\to} \mathbb{R}$ as \revision{$H_1(x, v) = \alpha h(x)\left(V(-x) -v\right)$},
where $h \colon \mathbb{R} \revision{\to} \mathbb{R}$ is a suitable
smooth and compactly supported function on $[-\epsilon_0,0]$ and weights the strength of the interaction depending on the distance between two vehicles. \revision{Typical choice, as in the case of flocking dynamics, is to consider a weighting function $h$ which is decreasing with respect to the distance, e.g.~$h(x)=\frac{1}{1+x^2}$.
\begin{figure}
  \centering
    \includegraphics[width=0.49\textwidth]{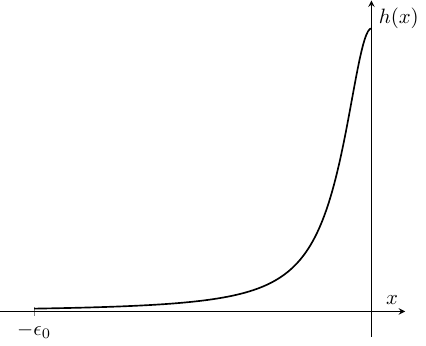}
    \caption{Schematic representation of the weighting function $h$.\label{fig:hfunc}}
\end{figure}
A graph is provided in Figure~\ref{fig:hfunc}. We observe that the introduction of the range of interaction $\epsilon_0$ allows each vehicle to interact with more than one vehicle ahead.}

Then, formally, \revision{ for each $i\in \{1, \dots, P\}$ and $(x_i, v_i) \in \mathbb{R}\times \mathbb{R}^{+}_0$}, the Bando-term in \eqref{eqn_bando_FtL} can be rewritten as
\begin{align*}
H_1*_1\mu_P(x_i, v_i) =& \frac{1}{P} \sum\limits_{k=1}^{P} H_1(x_i-x_k, v_i)
=\frac{\alpha}{P} \left(\sum \limits_{k \in i_{\varepsilon_0}}h(x_i-x_k) \left( V(x_k-x_i) - v_i\right)\right),
\end{align*}
where $*_1$ is the convolution with respect to the first variable, and $$i_{\varepsilon_0} = \{ k\revision{\in \{1, \dots, P\}} \colon 0<x_k - x_i< \varepsilon_0\}.$$
Similarly, define a convolution kernel $H_2 \colon \mathbb{R} \times \mathbb{R} \revision{\to} \mathbb{R}$ as $H_2(x,v) = \beta h(x) \frac{-v}{x^2}$. Then,  \revision{ for each $i\in \{1, \dots, P\}$ and $(x_i, v_i) \in \mathbb{R}\times \mathbb{R}^{+}_0$}, we formally rewrite the FtL-term of \eqref{eqn_bando_FtL} as
\begin{align*}
H_2*\mu_P(x_i, v_i) = & \frac{1}{P} \sum\limits_{k=1}^{P} H_2(x_i-x_k, v_i-v_k)
= \frac{\beta}{P} \left(\sum \limits_{k \in i_{\varepsilon_0}} h(x_i-x_k) \frac{v_k-v_i}{(x_i-x_k)^2}\right),
\end{align*}
where $*$ is the $(x,v)$-convolution.

Formally, the Bando-FtL model \eqref{eqn_bando_FtL} can be written using the convolution kernels as follows
\begin{equation}
\label{Bando-FTL convolution}
\left\{
\begin{array}{ll}
\dot{x_i} = v_i, & \\
\dot{v_i} = (H_1 *_1 \mu_P + H_2 * \mu_P) (x_i, v_i), & i\revision {\in \{} 1, \dots, P\revision{\}}.
\end{array}
\right.
\end{equation}

Model~\eqref{Bando-FTL convolution} has thus a close link to bounded confidence models for opinion formation, flocking and swarming behaviors~\cite{MotschTadmor2014}.

Next, we will focus also on descriptions based on PDEs. In particular, system \eqref{Bando-FTL convolution} formally admits the following mean-field limit as $P \to \infty$:
\begin{equation}\label{mean_field_limit_formally}
\partial_t \mu + v \partial_x \mu + \partial_v (\left(H_1*_1 \mu + H_2*\mu \right)\mu)=0,
\end{equation}
which gives a partial differential equation of Vlasov-type. Here $\mu$ represents the density distribution of the vehicles in position-velocity variables in a single lane. Equation \eqref{mean_field_limit_formally} describes the evolution of the density distribution $\mu$ with respect to time in the mesoscopic level. This can be easily derived in a formal way following classical computations, e.g.~see \cite{carrillo2010particle}, by considering a test function $\varphi \in C_0^1 (\mathbb{R}^{2})$ and computing the time derivative $\frac{\mathrm{d}}{\mathrm{d}t} \langle \mu_P(t), \varphi \rangle$. Mean-field limits can be also rigorously derived\revision{, for more information, please see }~\cite{ref8}.

\revision{In Section~\ref{sec_finite_dimen_hybrid} we specialize the microscopic Bando-FtL model to the case of multi-lane multi-class traffic flow. For the finite dimensional model, we will lead to a hybrid system, where the discrete events are determined by lane changes, and whose general definition is recalled in the next subsection. Whereas, in the mean-field limit of the multi-lane multi-class microscopic model, lane changes will cause presence of source terms in the Vlasov-type equation.}

\subsection{\revision{Hybrid control systems}} \label{ssec:hybrid}
\revision{
A hybrid control system is a generic term for such controlled system that involves continuous dynamics and discrete events. The discrete events are due to an automaton that contains a finite number of discrete states called locations. The continuous dynamics are given by the continuous time controlled system at each location. The following definition formalizes the details: 
\begin{definition} \label{def:hybridSystem}
A hybrid control system is a $6$-tuple $\Sigma = (\mathcal{L}, \mathcal{M}, U, \mathcal{U}, g, S)$ such that
\vspace{0.5em}
\begin{itemize}\setlength\itemsep{0.5em}
\item[(1)] $\mathcal{L}$ is the set of a finite number of discrete states, i.e.~the locations;
\item[(2)] $\mathcal{M} = \{\mathcal{M}_{ \ell}\}_{ \ell \in \mathcal{L}}$ is a finite family of smooth manifolds representing the state spaces of locations;
\item[(3)] $U=\{U_{\ell}\}_{\ell\in \mathcal{L}}$ is a finite family of sets representing the control space;
\item[(4)] $\mathcal{U} = \{\mathcal{U}_{\ell}\}_{\ell \in \mathcal{L}}$ is such that $\mathcal{U}_{\ell} = \{u: \text{Dom}(u)\subset \mathbb{R}_0^+ \to U_{\ell} \mbox{ measurable}\}$ for each $\ell \in \mathcal{L}$.
$\mathcal{U}_{\ell}$ represents the set of admissible controls at location $\ell$;
\item[(5)] $g= \{g_{\ell}\}_{\ell \in \mathcal{L}}$ is a family of maps, 
$g_{\ell} \colon \mathcal{M}_{\ell} \times \mathcal{U}_{\ell} \to T\mathcal{M}_{\ell}$, such that for every $(x,u) \in \mathcal{M}_{\ell} \times \mathcal{U}_{\ell}$, $g_{\ell}(x, u) \in T_x \mathcal{M}_{\ell}$. Here $T \mathcal{M}_{\ell}$ is the tangent bundle to the manifold $\mathcal{M}_{\ell}$, $T_x \mathcal{M}_{\ell}$ is the tangent space to $\mathcal{M}_{\ell}$ at $x$, and $g_{\ell}$ is the dynamical law at location $\ell$;
    \item[(6)] $S$ is a subset of $SW(\Sigma)$, where $SW(\Sigma) \coloneqq \{(\ell, x, \ell', x') \colon \ell, \ell' \in \mathcal{L}, x \in \mathcal{M}_{\ell}, x'\in \mathcal{M}_{\ell'}\}$.
\end{itemize}
\end{definition}
The hybrid states of the above hybrid control system $\Sigma$ is identified by a $2$-tuple, more precisely, we have the following definition,
\begin{definition} \label{def:hybridState}
    A hybrid state of the hybrid control system $\Sigma$ is a $2$-tuple $(\ell, x
    )$, where $\ell \in \mathcal{L}$ such that $x \in \mathcal{M}_{\ell} \in \mathcal{M}$. In the other words, the first variable of the state of the hybrid control system $\Sigma$ indicates the location, and the second variable indicates the space state of the location. 
    We denote by $\mathcal{HS}$ the set of the hybrid states of the hybrid control system $\Sigma$.
\end{definition}
Now we define the admissible hybrid trajectories of the hybrid control system $\Sigma$ with initial data $(\ell_0, x_0) \in \mathcal{L} \times \mathcal{M}_{\ell_0}$.
\begin{definition} \label{def:hybridTrajectory}
A map $\varphi \colon [a, b] \subset \mathbb{R}_0^+ \to \mathcal{HS},\ \varphi(t) = (\ell(t), x(t)
)$ is an admissible hybrid trajectory of the hybrid system $\Sigma=(\mathcal{J}, \mathcal{M}, U, \mathcal{U}, g, S)$ with the initial data $(\ell_0, x_0) \in \mathcal{L} \times \mathcal{M}_{\ell_0}$ if $\varphi(a) = (\ell(a), x(a)) = (\ell_0, x_0)$ and if there exists $s\in \mathbb{N}$, such that
\vspace{0.5em}
\begin{enumerate}\setlength\itemsep{0.5em}
\item[(1)] $a = t_0 < t_1 < \dots < t_s =b$;
\item[(2)] for $k=0, \dots, s-1$, $[t_{k}, t_{k+1}) \subset \text{Dom}(u_k)$, $u_k \in \mathcal{U}_{\ell(t_k)}$;
\item[(3)] for $k=0, \dots, s-1$, $\ell \restriction_{[t_k, t_{k+1}]}$ is constant;
\item[(4)] for $k=0, \dots, s-1$ the map $x \colon (t_k, t_{k+1}) \to \mathcal{M}_{\ell(t_k)}$ is absolutely continuous and $\lim\limits_{t \to t_{k+1}} x(t)$ exists;
\item[(5)] for $k=0, \dots, s-1$ and for almost every $t \in [t_k, t_{k+1}]$, we have $$\dot{x}(t) = g_{\ell(t_k)}(t,x(t), u_k(t));$$
\item[(6)] for $k=0, \dots, s-1$, $(\ell(t_{k}),x(t_{k}), \ell(t_{k+1}), x(t_{k+1})) \in S$.
\end{enumerate}
\end{definition}
Hybrid control systems have numerous applications industrial process control, manufacturing and robotics, automotive control and so on, see \cite{Tomlin_1998, Pepyne_2000, fornasier2014mean}
We mainly focus on the application of the hybrid control systems on multi-lane traffic flows in 
sections \ref{sec_finite_dimen_hybrid} and \ref{sec_mean_field_limit_dimen_hybrid}.}
\subsection{The Generalized Wasserstein Distance}
In this subsection, we recall the definitions and some properties related to the Wasserstein distance and the generalized Wasserstein distance. For a complete introduction to Wasserstein distance, see \cite{villani2003topics} and to generalized Wasserstein distance, see \cite{piccoli2014generalized}. 

Let $\mathcal{M}$ be the space of Borel measures with finite mass, $\mathcal{P}$ be the space of probability measures (the measures in $\mathcal{M}$ with unit mass) and $\mathcal{M}^{p}$ be the space of Borel measures with finite $p$-th moment on $\mathbb{R}^{d}$, where $d$ is the dimension of the space. 
We also denote with $\mathcal{M}_0^{ac}$ the subspace of $\mathcal{M}$ of measures that are with bounded support and absolutely continuous with respect to the Lebesgue measure. Given a measure $\mu \in \mathcal{M}$, we denote with $|\mu|\colon = \mu(\mathbb{R}^d)$ its mass. 
Given a Borel map $\gamma \colon \mathbb{R}^d \to \mathbb{R}^{d}$, the push-forward of $\mu$ by $\gamma$, $\gamma \# \mu$, is defined as for every Borel set $A \subset \mathbb{R}^d$, $ \gamma \# \mu (A) \colon = \mu(\gamma^{-1}(A))$. One can see that the mass of $\gamma \# \mu$ is identical to the mass of $\mu$, i.e., $|\mu| = |\gamma \# \mu|$.

Given two probability measures $\mu$, $\nu \in \mathcal{P}$, a probability measure $\pi$ on the product space $\mathbb{R}^d \times \mathbb{R}^d$ is said to be an admissible transference plan from $\mu$ to $\nu$ if the following properties hold: 
\begin{equation}
\label{transference_plan_1}
\int_{y\in \mathbb{R}^d} \,\mathrm{d}\pi(x,y) = \mathrm{d}\mu(x), \quad \int_{x\in \mathbb{R}^d} \,\mathrm{d}\pi(x,y)= \mathrm{d}\nu(y).
\end{equation}
We denote the set of admissible transference plans from $\mu$ to $\nu$ by $\Pi(\mu, \nu)$. Note that the set $\Pi(\mu, \nu)$ is always nonempty, since the tensor product $\mu \tens \nu \in \Pi(\mu, \nu)$. To each admissible transference plan from $\mu$ to $\nu$, $\pi$, one can define a cost as follows: 
$J[\pi]\colon = \int_{\mathbb{R}^d \times \mathbb{R}^d} |x-y|^p \, \mathrm{d}\pi(x,y)$,
where $|\cdot|$ represents the Euclidean norm. 
A minimizer of $J$ in $\Pi(\mu, \nu)$ always exists. Furthermore, 
the space of probability measures with finite $p$-th moment, $\mathcal{P} \cap \mathcal{M}^p$, is a natural space in which $J$ is finite. Thus for any two measures $\mu, \nu \in \mathcal{P} \cap \mathcal{M}^p$, one can define the following operator which is called Wasserstein distance 
$W_p(\mu, \nu) \colon = \left(\min\limits_{\pi\in \Pi(\mu, \nu)} J[\pi]\right)^{\frac{1}{p}}$.
Note that if $\nu^{m,1} = \frac{1}{m}\sum\limits_{k=1}^{m} \delta_{\xi_k^1}$ and $\nu^{m,2} = \frac{1}{m}\sum\limits_{k=1}^{m} \delta_{\xi_k^2}$ are two atomic measures with $m\in \mathbb{Z}^{+}$, $\xi_k^1, \xi_k^2 \in \mathbb{R}^{d}$, then 
$
W_1(\nu^{m,1}, \nu^{m,2}) \leq \frac{1}{m}\sum\limits_{k=1}^{m}|\xi_k^1-\xi_k^2|$.

We additionally recall the following lemmas related to Wasserstein distance (see, e.g., Lemma $3.11$, Lemma $3.13$, Lemma $3.15$, Lemma $4.7$ in \cite{ref8}).
\begin{lemma}
\label{lm_66}
Let $f_1$ and $f_2 \colon \mathbb{R}^n \to \mathbb{R}^{n}$ be two bounded Borel measurable functions. Then for every $\mu \in \mathcal{P}(\mathbb{R}^n)\cap \mathcal{M}^1(\mathbb{R}^n)$, one has 
\[W_1(f_1 \# \mu, f_2 \# \mu) \leq \|f_1 - f_2\|_{L^{\infty}(\supp \mu)}.\]
If in addition, $f_1$ is locally Lipschitz continuous, and $\mu, \nu \in \mathcal{P}(\mathbb{R}^n)\cap \mathcal{M}^{1}(\mathbb{R}^n)$ are both compactly supported on a ball $B$ of $\mathbb{R}^n$, then 
\[W_1(f_1 \# \mu, f_1\#\nu) \leq LW_1(\mu, \nu),\]
where $L$ is the Lipschitz constant of $f_1$ on $B$. 
\end{lemma}
Now we state the following assumption on map $H \colon \mathbb{R}^{2d} \to \mathbb{R}^{d}$: 
\begin{align*}
&(H1) \quad H \text{ is locally Lipschitz};\\
&(H2) \quad H \text{ is of sub-linear growth, that is, there exists a constant } C>0 \text{ such that } \\ 
& \quad |H(\xi)| \leq C(1+|\xi|), \text{ for all }\xi\in \mathbb{R}^{2d}.
\end{align*}
\begin{lemma}
\label{lm_67}
Let $H$ be a map satisfying condition (H1)-(H2). Let $R >0$. Let $\mu, \nu \colon [0, T] \to \mathcal{P}(\mathbb{R}^{2d})\cap \mathcal{M}^1(\mathbb{R}^{2d})$ be continuous maps with respect to the first order Wasserstein distance $W_1$ both satisfying 
\[\supp{\mu(t)} \subset B(0, R) \text{ and } \supp{\nu(t)} \subset B(0, R),\]
for every $t \in [0, T]$. Then for every $\rho>0$, there exists a constant $L_{\rho, R}$ such that 
\[\|H*\mu(t)-H*\nu(t)\|_{L^{\infty}(B(0, \rho))} \leq L_{\rho, R}W_1(\mu(t), \nu(t)).\]
\end{lemma}
Next, we recall the definition of the generalized Wasserstein distance on, $\mathcal{M}$, the space of Borel measures with finite mass on $\mathbb{R}^d$. For more detail, see \cite{piccoli2014generalized}. 
\begin{definition}
Given $a,b \in (0, \infty)$ and $p \geq 1$, the generalized Wasserstein distance between two measures $\mu, \nu \in \mathcal{M}^p$ is 
\begin{equation}
\label{defgwd}
W_{p}^{a,b}(\mu, \nu) \colon = \inf \limits_{\substack{\tilde{\mu}, \tilde{\nu} \in \mathcal{M}^{p} \\ |\tilde{\mu}|=|\tilde{\nu}|}}\left(a\left(|\mu-\tilde{\mu}|+|\nu - \tilde{\nu}|\right)+bW_p(\tilde{\mu}, \tilde{\nu})\right).
\end{equation}
\end{definition}
\begin{remark}
The standard Wasserstein distance is defined only for probability measures. Combining the standard Wasserstein distance and $L^1$ distance, the generalized Wasserstein distance can be applied to measures with different masses. 
\end{remark}
If $\mu_1$ is absolutely continuous with respect to $\mu \in \mathcal{M}$ and for every Borel set $A \subset \mathbb{R}^d$, $\mu_1(A) \leq \mu(A)$, then we write $\mu_1 \leq \mu$. 
\begin{remark}
The infimum in equation \eqref{defgwd} is always attained if one restrict the computation in equation \eqref{defgwd} to $\tilde{\mu} \leq \mu$, $\tilde{\nu} \leq \nu$. 
\end{remark}
We recall some simple properties of the generalized Wasserstein distance, $W_p^{a,b}$. Compare the following proposition with Proposition 2 in \cite{piccoli2014generalized}. 
\begin{proposition}
\label{pro_gwd_property}
Let $\mu, \nu, \mu_1, \mu_2, \nu_1, \nu_2$ be measures in $\mathcal{M}^p$. The following properties of the generalized Wasserstein distance $W_1^{1,1}$ hold: 
\begin{align*}
& W_{1}^{1,1}(k\mu, k\nu) \leq kW_1^{1,1}(\mu, \nu) \text{ for } k \geq 0;\\
& W_{1}^{1,1}(\mu_1+\mu_2, \nu_1+\nu_2) \leq W_1^{1,1}(\mu_1, \nu_1)+W_1^{1,1}(\mu_2, \nu_2). 
\end{align*}
\end{proposition}
Similar to Lemmas \ref{lm_66}, \ref{lm_67}, we have the following lemmas for the generalized Wasserstein distance. 
\begin{lemma}
\label{lm_66_gwd}
Let $f_1, f_2 \colon \mathbb{R}^{n} \to \mathbb{R}^{n}$ be bounded Borel measureable functions. Then for every $\mu \in \mathcal{M}^1(\mathbb{R}^{n})$, one has 
\[W_1^{1,1}(f_1\#\mu, f_2\#\mu) \leq \|f_1 - f_2\|_{L^{\infty}(\supp\mu)}.\]

If in addition $f_1$ is locally Lipschitz continuous Borel measurable functions, then for $\mu, \nu \in \mathcal{M}^1(\mathbb{R}^n)$ compactly supported on a ball $B$ of $\mathbb{R}^n$, 
\[W_1^{1,1}(f_1\#\mu, f_1\#\nu) \leq \max\{L, 1\} W_1^{1,1}(\mu, \nu),\]
where $L$ is the Lipschitz constant of $f_1$ on $B$. 
\end{lemma}
\begin{lemma}
\label{lm_67_gwd}
Let $H$ be a map satisfying condition (H1)-(H2). Let $R >0$ be fixed. Let $\mu, \nu \colon [0, T] \to \mathcal{M}^1(\mathbb{R}^{2d})$ be continuous maps with respect to the generalized Wasserstein distance $W_1^{1,1}$ both satisfying 
\[\supp{\mu(t)} \subset B(0, R) \text{ and } \supp{\nu(t)} \subset B(0, R),\]
for every $t \in [0, T]$. Then for every $\rho>0$, there exists a constant $L_{\rho, R}$ such that 
\begin{equation}
\label{eqn_Lem2.8}
\|H*\mu(t)-H*\nu(t)\|_{L^{\infty}(B(0, \rho))} \leq L_{\rho, R}W_1^{1,1}(\mu(t), \nu(t)).
\end{equation}
\end{lemma}

One can prove Lemma \ref{lm_66_gwd} and Lemma \ref{lm_67_gwd} by combining Lemma \ref{lm_66}, Lemma \ref{lm_67}, and the definition of generalized Wasserstein distance. 
\subsection{ Carath\'eodory Differential Equations}
In this section, we recall 
the following global existence and uniqueness result (see Theorem $6.2$ in~\cite{fornasier2014mean}) for\\
Carath\'eodory 
differential equations. For further detailed discussions, see also~\cite{filipov1988differential}. 
\begin{theorem}
\label{global_existence_uniqueness}
Let $T>0$ and $n \geq 1$ be fixed. Consider a Carath\'eodory function $g \colon [0, T] \times \mathbb{R}^n \to \mathbb{R}^n$. Assume that there exists a constant $C>0$ such that for almost every $t\in [0, T]$ and every $y \in \mathbb{R}^n$, $|g(t, y)| \leq C(1+|y|)$.
Then given $y_0 \in \mathbb{R}^n$, there exists a solution $y(t)$ of 
$
\dot{y}(t) = g(t, y(t))
$
on the whole interval $[0, T]$ such that $y(0) = y_0$. Any such solution satisfies for every $t \in [0, T]$,
$|y(t)| \leq \left(|y_0|+Ct\right)e^{Ct}$.

If in addition, for every relatively compact open subset of $\mathbb{R}^n$, 
$
|g(t, y_1) - g(t, y_2)| \leq L|y_1-y_2|
$
holds, then the solution is uniquely determined by the initial condition $y_0$ on the whole interval $[0, T]$. 
\end{theorem}

\subsection{Partial Differential Equations of Vlasov-type with Source Term}
\label{sec_PDE_source}
In this subsection, we consider partial differential equations of Vlasov-type. 

Let $H_1, H_2$ be two maps satisfying condition (H1)-(H2). Let $T>0$, $R>0$ be fixed. Consider a continuous map $\mu \colon [0, T] \to \mathcal{P}(\mathbb{R}^{2d}) \cap \mathcal{M}^1(\mathbb{R}^{2d})$ with respect to the first order Wasserstein distance, $W_1$, such that $\supp \mu(t) \subset B(0, R)$ for all $t \in [0, T]$, and a time dependent atomic measure $\nu(t)(y,w) = \frac{1}{M}\sum\limits_{k=1}^{M} \delta_{\left(y_k(t), w_k(t)\right)}$ supported on the absolutely continuous trajectories $t \mapsto (y_k(t), w_k(t))$, $k=1, \dots, M_j$. Then given an initial datum $P_0 \colon = (x_0, v_0) \in \mathbb{R}^{2d}$, there exists a unique solution $P(t) \colon=(x(t), v(t))$ on the whole time interval $[0, T]$ to the following system of ODEs on~$\mathbb{R}^{2d}$ 
\[ \left\{
\begin{array}{ll}
\dot{x}(t) = v(t)\\
\dot{v}(t) = \left(H_1*_1(\mu + \nu)+H_2*(\mu+\nu)\right)(x(t), v(t)). 
\end{array} 
\right. \]
Therefore, one can consider a family of flow maps
\begin{equation}
\label{eqn_flow}
\mathcal{T}_{t}^{\mu,\nu}\colon P_0 \in \mathbb{R}^{2d} \mapsto P(t) \in \mathbb{R}^{2d}.
\end{equation}
indexed by $t \in [0, T]$. Furthermore, the flow map $\mathcal{T}_t^{\mu, \nu}$ is Lipschitz continuous. In fact, let 
$\mu^q \colon [0, T] \to \mathcal{P}(\mathbb{R}^{2d}) \cap \mathcal{M}^{1}(\mathbb{R}^{2d})$, $q=1,2$, be two continuous maps with respect to Wasserstein distance and be equi-compactly supported in $B(0, R)$. Let $\nu^{1}$, $\nu^{2}$ be two atomic measures supported on the respective absolutely continuous trajectories $t \mapsto (y_k^{q}(t), w_k^{q}(t))$, $q=1,2$ and $k = 1, \dots , M$. Fix $r>0$. Then there exist constants $\rho, L>0$, such that
whenever $|P_1| \leq r$,$|P_2| \leq r$,
\begin{equation}
\label{eqn_flow_bounded_Lip}
\begin{aligned}
|\mathcal{T}_{t}^{\mu^1, \nu^{1}}(P_1) - \mathcal{T}_{t}^{\mu^2, \nu^{2}}(P_2)| \leq& e^{Lt}|P_1-P_2| \\ 
& + \int_{0}^{t}e^{L(s-t)}\left\|\left(H_1*_1(\mu^1+\nu^{1})+H_2*(\mu^1+\nu^{1})\right) \right.\\ 
&\left.-\left(H_1*_1(\mu^2+\nu^{2})+H_2*(\mu^2+\nu^{2})\right) \right\|_{L^{\infty}(B(0, \rho))}\,\mathrm{d}s,
\end{aligned}
\end{equation}
for every $t \in [0, T]$. For more details, please see \cite{fornasier2014mean}. 

Given an initial condition $\mu_0 \in \mathcal{P}(\mathbb{R}^{2d}) \cap \mathcal{M}^1(\mathbb{R}^{2d})$ of bounded support, we say that a measure $\mu(t)$ is a weak equi-compactly supported solution of the following Vlasov-type PDE with the initial datum $\mu_0$,
\begin{equation}
\label{PDE_Vlasov}
\partial_t \mu + v\cdot \nabla_{x} \mu + \nabla_{v} \cdot \left[(H_1*_1(\mu+\nu)+H_2*(\mu+\nu))\mu\right]=0, 
\end{equation} 
if
$
(i) \quad \mu(0) = \mu_0$; \\
$(ii)\quad \supp\mu(t) \subset B(0, R) \text{ for all } t\in [0, T], \text{ for some } R>0$;\\
$(iii)\quad\text{for every } \varphi \in C_c^{\infty} (\mathbb{R}^{2d})$, 
\[\frac{\mathrm{d}}{\mathrm{d}t} \int_{\mathbb{R}^{2d}} \varphi(x,v)\,\mathrm{d}\mu(t)(x,v) = \int_{\mathbb{R}^{2d}} \nabla \varphi(x,v) \cdot \tilde{\omega}_{H_1, H_2,\mu, \nu^j}(t, x, v)\,\mathrm{d}\mu(t)(x,v)\]
where $\tilde{\omega}_{H_1, H_2,\mu, \nu}(t, x, v)\colon [0,T] \times \mathbb{R}^d \times \mathbb{R}^d \to \mathbb{R}^{2d}$ is defined as 
\begin{equation}
\label{eqn_omega}
\tilde{\omega}_{H_1, H_2,\mu, \nu}(t, x, v)\colon = (v, (H_1*_1(\mu+\nu)+H_2*(\mu+\nu))(x,v)). 
\end{equation}
Furthermore, following from Section $8.1$ in \cite{ambrosio2008gradient}, a measure $\mu(t)$ is a weak equi-compactly supported solution of equation \eqref{PDE_Vlasov} if and only if it satisfies condition $(ii)$ and the measure-theoretical fixed point equation 
$ \mu(t)= \left(\mathcal{T}_t^{\mu, \nu}\right)\#\mu_0$ 
where the flow function $\mathcal{T}_t^{\mu, \nu}$ is defined in equation \eqref{eqn_flow}. 

Now we consider solutions to the following Vlasov-type PDE with initial datum $\mu_0 \in \mathcal{M}_0^{ac}(\mathbb{R}^{2d})\cap \mathcal{M}^{1}(\mathbb{R}^{2d})$ and source term $S$
\begin{equation}
\label{eqn_source}
\partial_t \mu + v\cdot \nabla_{x} \mu + \nabla_{v} \cdot \left[(H_1*_1(\mu+\nu)+H_2*(\mu+\nu))\mu\right]=S(\mu)
\end{equation}
under the following hypotheses: 
\begin{align*}
& (S_1) \quad S(\mu) \text{ has uniformly bounded mass and support, that is, there exist } Q, R,\\
& \quad \quad\quad \text{ such that } S(\mu)(\mathbb{R}^{2d}) \leq Q, \text{ and } \supp(S(\mu)) \subset B(0, R);\\ 
& (S_2)\quad S \text{ is Lipschitz, that is, there exists } L, \text{ such that, for any } \mu, \nu \in \mathcal{M}^{1}(\mathbb{R}^{2d}),\\
& \quad \quad\quad W_{1}^{1,1}(S(\mu), S(\nu)) \leq L W_1^{1,1}(\mu, \nu). 
\end{align*}

A measure $\mu(t)$ is a weak solution of equation \eqref{eqn_source} with a given initial datum $\mu_0 \in \mathcal{M}_0^{ac}(\mathbb{R}^{2d})\cap \mathcal{M}^{1}(\mathbb{R}^{2d})$, if $\mu(0) = \mu_0$ and if for every $\varphi \in C_c^{\infty}(\mathbb{R}^{2d})$, it holds

\begin{align*}
& \frac{\mathrm{d}}{\mathrm{d}t} \int_{\mathbb{R}^{2d}} \varphi(x,v)\,\mathrm{d}\mu(t)(x,v) =\\
=& \int_{\mathbb{R}^{2d}} \varphi(x,v)\,\mathrm{d}S(\mu)(x,v)+ \int_{\mathbb{R}^{2d}} \nabla \varphi(x,v) \cdot \tilde{\omega}_{H_1, H_2,\mu, \nu}(t, x, v)\,\mathrm{d}\mu(t)(x,v),
\end{align*}
where $\tilde{w}_{H_1, H_2, \mu, \nu}$ is as defined in \eqref{eqn_omega}. 
\begin{theorem}
Given an initial datum $\mu_0 \in \mathcal{M}_0^{ac}(\mathbb{R}^{2d})\cap \mathcal{M}^{1}(\mathbb{R}^{2d})$, there exists a unique weak solution $\mu(t)$ to equation \eqref{eqn_source} under the hypotheses $(S_1), (S_2)$. Furthermore, $\mu(t) \in \mathcal{M}_0^{ac}(\mathbb{R}^{2d})\cap \mathcal{M}^{1}(\mathbb{R}^{2d})$.
\end{theorem}
One can construct a weak solution $\mu(t)$ to equation \eqref{eqn_source} based on a Lagrangian scheme by sample-and-hold. Given a fixed $k \in \mathbb{N}^{+}$, define $\Delta t \colon = \frac{T}{2^k}$ and decompose the time interval $[0,T]$ in $[0, \Delta t], [\Delta t, 2\Delta t],\dots, [(2^k-1)\Delta t, 2^k\Delta t]$. We define\\ 
$
\mu_k(0) \colon = \mu_0$;\\
$\mu_k((n+1)\Delta t) \colon = \mathcal{T}_{\Delta t}^{\mu_k(n \Delta t), \nu(n\Delta t)}\# \mu_k(n\Delta t) + \Delta t S(\mu_k(n \Delta t))$;\\
$\mu_k(t) \colon = \mathcal{T}_{\tau}^{\mu_k(n\Delta t), \nu(n\Delta t)}\#\mu_k(n \Delta t) + \tau S(\mu_k(n \Delta t))
$,\\
where $n$ is the maximum integer such that $t - n\Delta t \geq 0$
and $\tau \colon = t - n \Delta t$.
Then $\mu(t)= \lim \limits_{k \to \infty} \mu_k(t)$ is the unique weak solution to equation \eqref{eqn_source}. For more detail, please see \cite{piccoli2014generalized}.

\subsection{A Revised Version of Arzelà–Ascoli Theorem} In this subsection, 
we will provide a proof to a revised version of Arzelà–Ascoli theorem. 
\begin{theorem}
\label{revised_AA}
Let $K$ be a compact subset of $\mathbb{R}$ and let $D$ be a complete and totally bounded metric space with metric $d$. Consider a sequence of functions $\{f_n\}_{n=1}^{\infty}$ in $C(K; D)$. If there exists $L>0$, such that the following is true: for any $\varepsilon>0$, there exists $N>0$, such that, whenever $n \geq N$, 
\[d(f_n(t), f_n(s)) \leq L|t-s| + \min\{\varepsilon, |t-s|\}, \forall s, t \in K\]
then the sequence $\{f_n\}_{n=1}^{\infty}$ has a uniformly convergent sub-sequence. 
\end{theorem}
\begin{proof}
First note that the subset $S=K\cap \mathbb{Q}$ of set $K \subset \mathbb{R}$ is countable and dense, that is, $K$ is separable. We list the countably many elements of $S$ as $\{t_1, t_2, t_3, \dots\}$. 

We will find a sub-sequence of $\{f_n\}$ that converges point-wise on $S$ by a standard diagonal argument. 

Since $D$ is complete and totally bounded, $D$ is sequentially compact. Thus the sequence $\{f_n(t_1)\}_{n=1}^{\infty}$ in $D$ has a convergent sub-sequence, which we will write using double subscripts, $\{f_{1,n}(t_1)\}_{n=1}^{\infty}$. Similarly, the sequence $\{f_{1,n}(t_2)\}_{n=1}^{\infty}$ also has a convergent sub-sequence $\{f_{2,n}(t_2)\}_{n=1}^{\infty}$. By proceeding in this way, we obtain a countable collection of sub-sequences of our original sequence $\{f_{n}\}_{n=1}^{\infty}$: 

\[\begin{array}{cccc}
f_{1,1} & f_{1,2} & f_{1,3}, & \dots\\
f_{2,1} & f_{2,2} & f_{2,3}, & \dots\\
f_{3,1} & f_{3,2} & f_{3,3}, & \dots\\
\vdots & \vdots & \vdots & \dots
\end{array}
\]
where the sequence in the $n$-th row converges at the points $t_1, t_2, \dots, t_n$, and each row is a sub-sequence of its previous row. Let $\{g_n\}$ be the diagonal sequence produced in the previous step, i.e., $g_n = f_{n,n}$ for each $n \in \mathbb{N}$. Then the sequence $\{g_{n}\}$ is a sub-sequence of the original sequence $\{f_n\}$ that converges at each point of $S$. 

Next, we will show that the sub-sequence $\{g_n\}$
of $\{f_n\}$ is uniformly convergent. 
Let $\varepsilon>0$ be given and choose $\delta = \min\left\{\frac{\varepsilon}{6L}, \frac{\varepsilon}{6}\right\}$. Then there exists $N_1>0$, such that for every $n \geq N_1$, and for any $s,t \in K$ with $|s-t|<\delta$, 
\begin{align*}
d(g_n(t), g_n(s)) & \leq L|t-s| + \min\{\frac{\varepsilon}{6}, |t-s|\}
\leq L\delta + \frac{\varepsilon}{6} \leq \frac{\varepsilon}{3}.
\end{align*}

Since $K$ is compact, for any positive integer $M>\frac{1}{\delta}$, there exists a finite set $S_M \subset S$ such that 
$K \subset \bigcup\limits_{s\in S_M}B_{\frac{1}{M}}(s)$. Since the sequence $\{g_n\}$ converges at each point of $S_M$, there exists $N_2>0$, such that whenever $n, m > N_2$,
\[d(g_n(s), g_m(s)) < \frac{\varepsilon}{3}, \quad \forall s \in S_M.\]
Let $t \in K$ be arbitrary but fixed. Then there exists some $s \in S_M$ such that $|s-t|< \delta$. In addition, let $N = \max\{N_1, N_2\}$. Then whenever $n ,m > N$, 
\begin{align*}
d(g_n(t),g_m(t)) & \leq d(g_n(t), g_n(s))+d(g_n(s), g_m(s))+d(g_m(s), g_m(t))< \epsilon.
\end{align*}
Hence the sub-sequence $\{g_n\}$ of the original sequence $\{f_n\}$ is uniformly Cauchy. Since the metric space $D$ is complete, $C(K; D)$ is complete with respect to the uniform metric. Thus the sub-sequence $\{g_n\}$ is uniformly convergent. 
\end{proof}

\section{The Finite-Dimensional Hybrid System}
\label{sec_finite_dimen_hybrid}
In this section, we specify the Bando-FtL model introduced in Section~\ref{sec_models} to the case of \revision{multi-lane and multi-class vehicles with} lane changing maneuvers\revision{, leading to a finite-dimensional hybrid system, cf.~Section~\ref{ssec:hybrid}. Then, we study existence and uniqueness of solutions to the resulting hybrid system.}

\subsection{\revision{The model}}
In the case of multi-lane traffic, vehicles travel along multiple lanes with the possibility to change lane paying a cost related to such maneuver. We consider $m$ lanes and assume that $j \in J =\{1, \dots, m\}$ is the index of lanes. Now, to each vehicle $i$, we associate \revision{an extended vector of indices $\iota(i) = (i, j, \underline{i}_L, \underline{i}_F)$, where $j \in J$ is the lane index of vehicle $i$, while $\underline{i}_L=(i_L^j,i_L^{j+1},i_L^{j-1})$ and $\underline{i}_F=(i_F^j,i_F^{j+1},i_F^{j-1})$ are the vectors of the leader and follower indices, respectively, of vehicle $i$ on its current lane $j$ and of an hypothetical vehicle with same position as $i$ but on the two adjacent lanes $j+1, j-1$ (to the left and right of lane $j$, respectively.) If there is no lane to the left or to the right, we assume to label the corresponding index as $0$. Similarly, in the case of no leader or follower.\\
\indent
Let $(x_i,v_i)\in\mathbb{R}\times\mathbb{R}^+_0$ be again the space-velocity variables of the vehicle $i$. Similarly to~\eqref{L_F_labels}, we define $i_L^j$ and $i_F^j$ as
\begin{equation} \label{L_F_labels_multilane1}
i_L^{j} = \underset{\substack{k\in\{1,\dots,P\}\\ \pi_2(\iota(k))=j\\x_k> x_i}}{\arg\min} x_k-x_i, \quad i_F^{j} = \underset{\substack{k\in\{1,\dots,P\}\\ \pi_2(\iota(k))=j\\x_k< x_i}}{\arg\min} x_i-x_k,
\end{equation}
where $\pi_2$ is the projection of $\iota(i)$ on its second argument.
Instead, for $j'=j,j+1,j-1$, we define
\begin{equation} \label{L_F_labels_multilane}
i_L^{j'} = \underset{\substack{k\in\{1,\dots,P\}\\ \pi_2(\iota(k))=j'\\x_k\geq x_i}}{\arg\min} x_k-x_i, \quad i_F^{j'} = \underset{\substack{k\in\{1,\dots,P\}\\ \pi_2(\iota(k))=j'\\x_k\leq x_i}}{\arg\min} x_i-x_k.
\end{equation}
We observe that with the previous definition we identify the leader and the follower in lane $j'$ being the same if there is a vehicle next to $i$ having the same $x$ position.}

Each individual vehicle has a continuous dynamic governed by system \eqref{Bando-FTL convolution} before performing lane changing. Discrete dynamics of the vehicles will be generated due to lane changing. The presence of both continuous dynamics and discrete dynamics leads us to consider hybrid system, see \cite{ garavello2005hybrid, piccoli1998hybrid}.

In particular, in the following we consider two classes of vehicles and split the population of $P$ vehicles into $M$ autonomous vehicles and $N$ human-driven vehicles on an open stretch of road with $m$ lanes. We let $M_j$ and $N_j$ be the number of autonomous vehicles and the number of human-driven vehicles on lane $j \in J=\{1, \dots, m\}$, respectively. Then 
$\sum \limits_{j=1}^{m} M_j = M \text{ and } \sum\limits_{j=1}^{m} N_j = N$.

First, we study the dynamics of the $M+N$ vehicles from the microscopic point of view. As in \cite{fornasier2014mean}, we assume that we have \revision{a large amount $N$ of human-driven vehicles and} a small amount $M$ of autonomous vehicles that have a great influence on the population. \revision{This influence is modelled by controlled dynamics for the $M$ autonomous vehicles.}

\revision{In order to specify the continuous and discrete dynamics of the vehicles, we introduce the following labeling: the autonomous vehicles are identified by labels $i\in\{1,\dots,M\}$, whereas the human-driven vehicles are identified by labels $i\in\{M+1,\dots,M+N\}$. Furthermore,} we consider the following atomic measures in $\mathcal{M}^{+}(\mathbb{R}\times\mathbb{R}^{+})$ on \revision{ lane $j$}
\revision{
\begin{align}
\label{atomic_nu}
&\mu_{N_j} (t)= \frac{1}{N_j} \sum_{\substack{i\in\{M+1,\dots,M+N\}\\ \pi_2(\iota(i))=j}} \delta_{(x_i(t), v_i(t))},
\quad \nu^{j} (t)= \frac{1}{M_j} \sum_{\substack{i\in\{1,\dots,M\}\\ \pi_2(\iota(i))=j}} \delta_{(x_i(t), v_i(t))}.
\end{align}
}

\subsubsection{\revision{The continuous dynamics}}

From the microscopic point of view, the dynamics of \revision{each vehicle $i$} on lane $j \in J$\revision{, i.e.~with $\pi_2(\iota(i))=j$,} without lane changing are
\revision{
\begin{equation}\label{ODEs}
\begin{aligned}
\dot{x}_i &= v_i,\\
\dot{v}_i &= \left(H_1*_1(\mu_{N_j} + \nu^{j})+H_2*(\mu_{N_j}+\nu^j)\right)(x_i, v_i) + u_i^j, 
\end{aligned}
\end{equation}
where $u_i^j \colon [0, T] \to \mathbb{R}$, and $u_i^j \equiv 0$ if $i\in \{M+1, \dots, M+N\}$. The control term $u_i^j$, related to a car $i$ on lane $j$, is introduced in order to differ the dynamics of the two populations of vehicles. The introduction of the control term models the controlled behavior of the autonomous vehicles. The control can be chosen in order to reduce the instability of the effects in traffic flow, see \cite{GongPiccoliEtAl2020}. } \revision{As observed in Section~\ref{sec_models}, the explicit formulation of the convolution kernels $H_1$ and $H_2$ depends on the modeling of the follow-the-leader term and of the Bando term.
Here, we keep them general and we require that $H_1 \colon \mathbb{R}\times \mathbb{R}^{+} \to \mathbb{R}$, $H_2\colon \mathbb{R}\times \mathbb{R} \to \mathbb{R}$ are locally Lipschitz convolution kernels with sub-linear growth.} Particularly, there exists a constant $C>0$ such that for all $(x_1, v_1) \in \mathbb{R}\times \mathbb{R}^{+}$ and $(x_2, v_2) \in \mathbb{R} \times \mathbb{R}$, 
\begin{equation}
\label{sublinear}
|H_1(x_1,v_1)| \leq C(1+|(x_1, v_1)|) \text{ and } |H_2(x_2,v_2)| \leq C(1+|(x_2, v_2)|).
\end{equation}

\subsubsection{\revision{Lane changing maneuvers}}
Let $\Delta>0$ be fixed. Vehicle $n$ on $j \in J$ lane will perform \revision{a} lane \revision{changing to $j' =j+1\in J \text{ or } j'=j-1\in J$} lane at time $t \in [0, T]$ if the following conditions occur:
\begin{align*}
&\text{Safety: } \bar{a}_n^{j'}(t) \geq -\Delta \text{ and } \revision{\bar{a}_{n_F^{j'}}^{j'}(t) \geq -\Delta};
&\text{Incentive: } \bar{a}_n^{j'}(t) \geq a_n^j(t)+\Delta.
\end{align*}
\revision{Furthermore, we assume that $p([\min\{\bar{a}_n^{j'},\bar{a}_{n_F^{j'}}^{j'}\} + \Delta]_{+},[\bar{a}_n^{j'}-a_n^j - \Delta]_{+})$ represents the probability of vehicle $n$ performing lane changing from lane $j$ to $j'$, where $p\colon \mathbb{R} \times \mathbb{R} \to [0,1]$ is an increasing function with respect to both variables.}
Here \revision{$n_F^{j'}$} is the index of the first vehicle following vehicle $n$ on the new lane if vehicle $n$ changes lane at time $t$, $a_n^j(t)$ is the actual acceleration of vehicle $n$ at time $t$ on the $j$-th lane, $\bar{a}_n^{j'}(t)$ and \revision{$\bar{a}_{n_F^{j'}}^{j'}(t)$} are the expected accelerations of vehicle $n$ \revision{and ${n_F^{j'}}$, respectively, on lane $j'$} if vehicle $n$ changes lane at time $t$. For instance, if vehicle $n$ is an autonomous vehicle\revision{, i.e.~$n\in\{1,\dots,M\}$ with $\pi_2(\iota(n))=j$,} and if vehicle \revision{$n_F^{j'}$} is a human-driven vehicle, then at time $t \in [0, T]$,
\revision{
\begin{align*}
a_n^j(t) &= \left(H_1*_1(\mu_{N_j} (t)+ \nu^{j}(t))+H_2*(\mu_{N_j}(t)+\nu^j(t))\right)(x_n(t), v_n(t))+u_n^j(t), \\
\bar{a}_n^{j'}(t) &= \left(H_1*_1(\mu_{N_{j'}} (t)+ \bar{\nu}^{j'}(t))+H_2*(\mu_{N_{j'}}(t)+\bar{\nu}^{j'}(t))\right)(x_n(t), v_n(t))+u_n^{j'}(t),\\
\bar{a}_{n_F^{j'}}^{j'}(t) &= \left( H_1 *_1(\mu_{N_{j'}}(t)+\bar{\nu}^{j'}(t)) + H_2 *( \mu_{N_{j'}}(t)+\bar{\nu}^{j'}(t)) \right) (x_{n_F^{j'}}(t), v_{n_F^{j'}}(t)),
\end{align*}
where \[\bar{\nu}^{j'} = \frac{1}{M_{j'}+1} \left(\sum\limits_{\substack{\ell\in\{1,\dots,M\}\\ \pi_2(\iota(\ell))=j'}} \delta_{(x_\ell(t), v_\ell(t))}+\delta_{\left(x_n(t), v_n(t)\right)}\right).\]}

\revision{We discuss here the choices of the lane changing rules. The safety condition models the situation in which the reference vehicle $i$ will perform a lane change if there is enough space not to cause an extreme deceleration of itself and of its follower on lane $j'$. Instead, the incentive condition models the situation in which the reference vehicle $i$ will perform a lane change if difference of the accelerations on the new lane $j'$ and on the current lane $j$ is larger than the threshold $\Delta$. By considering acceleration for the lane changing condition, a vehicle needs to take into account simultaneously the space gap, velocity and velocity difference with its leading vehicle on the current and adjacent lane. Similar choices have been already considered in the literature, e.g.~see~\cite{KTH07}.}

\revision{Throughout the paper we make use of the following assumption on the lane changing maneuvers.
\begin{assumption}
We assume that there are no two vehicles changing lane at the same time. We assign each vehicle a timer over the whole time interval $[0, T]$. Let $N_{\tau} \in \mathbb{Z}^{+}$ be large and fixed and let $T_1= \frac{T}{N_{\tau}}$. A vehicle would consider changing lane only when its timer reaches to $T_1$. Here $T_1$ is called timer limit for all vehicles. Specifically, the timer $\tau_i$ for vehicle $i$ satisfies 
$\dot{\tau}_i =1, \tau_{i}(0) = \tau_{i,0} \in [0, T_1)$. In addition, the following is true: 
\begin{equation}\label{initial timer} 
\begin{aligned}
& \text{ if } k_1 \not = k_2 \in \{1, \dots, M+N\}, \text{ then } \tau_{k_1, 0} \not = \tau_{k_2, 0}.
\end{aligned}
\end{equation}
Besides, we reset the timer for each vehicle to be zero when it reaches to $T_1$.
\end{assumption}
Although the previous assumption might look restrictive, we point out that experimental studies have shown that lane changing are not frequent in real traffic. Therefore, the probability of having two vehicles changing lane at the same time is low. However, considering models for lane changes are important since it has been analyzed using safety measures, such as time-to-collision, that traffic safety is influenced by the flow across lanes. We refer to~\cite{dataset,HertyVisconti2018} for detailed discussions. Furthermore, we observe that the lane changing frequency can be arbitrarily increased by taking small values of $T_1$.}

\revision{From the mathematical point of view, the introduction of the cool down time $T_1$ will allow us to study qualitative properties over small time intervals when there is no lane changing. Since the number of lane changing is finite and lane changing only may occur at specific time, one can extend the small time interval into the whole time interval $[0, T]$ by repeating the procedure finitely many times. 
}
\revision{
\begin{example} \label{exampleMultilane}
    We present here a summary example of the basic working principles of the present finite-dimensional hybrid system.
    For the sake of simplicity, we restrict the focus on the case of a highway with two lanes, i.e.~$m=2$ and $J=\{1,2\}$. We assume that there are five vehicles on the road, i.e.~$P=5$, three of them flowing on lane $1$ and two on lane $2$. We refer to Figure~\ref{fig:schematicHybrid} for a schematic representation of the case study.
    \begin{figure}
        \centering
        \includegraphics[width=0.49\textwidth]{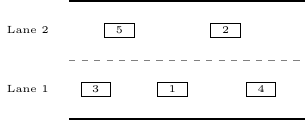}
        \caption{Schematic representation of the Example~\ref{exampleMultilane}.\label{fig:schematicHybrid}}
        \label{fig:my_label}
    \end{figure}
    The reference vehicle is identified by the label $i=1$ and travels on lane $j=1$. Its vector of indices defined by $\iota(1)=(1,1,\underline{1}_L,\underline{1}_F)$, with $\underline{1}_L=(4,2,0)$ and $\underline{1}_F=(3,5,0)$. In fact, vehicle labeled by $4$ is the leader in lane $1$, whereas vehicle labeled by $2$ is the leader in lane $2$. We observe that the last entry of $\underline{1}_L$ is $0$ because there is no right lane for the reference vehicle $1$, and so no leader on the right. Furthermore, and without loss of generality, we assume that the reference vehicle is the only autonomous on the road, namely $N=4$, with $N_1=N_2=2$, and $M=1$, with $M_1=1$, $M_2=0$. The vectors associated to the other, human-driven, vehicles are
    \begin{align*}
        \iota(2) &= (2,2,\underline{2}_L,\underline{2}_F), \quad \underline{2}_L = (0,0,4), \ \underline{2}_F = (5,0,1) \\
        \iota(3) &= (3,1,\underline{3}_L,\underline{3}_F), \quad \underline{3}_L = (1,5,0), \ \underline{3}_F = (0,0,0)\\
        \iota(4) &= (4,1,\underline{4}_L,\underline{4}_F), \quad \underline{4}_L = (0,0,0), \ \underline{4}_F = (1,0,2) \\
        \iota(5) &= (5,2,\underline{5}_L,\underline{5}_F), \quad \underline{5}_L = (2,0,1), \ \underline{5}_F = (0,0,3)
    \end{align*}
    Lane changing conditions for the reference vehicle $1$ become
    \begin{align*}
        &\text{Safety: } \ \bar{a}_1^{2}(t) \geq -\Delta \text{ and } \bar{a}_{5}^{2}(t) \geq -\Delta;
        &\text{Incentive: } \ \bar{a}_1^{2}(t) \geq a_1^1(t)+\Delta,
    \end{align*}
    where the acceleration functions are defined by
    \begin{align*}
        a_1^1(t) &= \left(H_1*_1(\mu_{N_1} (t)+ \nu^{1}(t))+H_2*(\mu_{N_1}(t)+\nu^1(t))\right)(x_1(t), v_1(t))+u_1^1(t), \\
        \bar{a}_1^{2}(t) &= \left(H_1*_1(\mu_{N_{2}} (t)+ \tilde{\nu}^{2}(t))+H_2*(\mu_{N_{2}}(t)+\bar{\nu}^{2}(t))\right)(x_1(t), v_1(t))+u_1^{2}(t),\\
        \bar{a}_{5}^{2}(t) &= \left( H_1 *_1(\mu_{N_{2}}(t)+\tilde{\nu}^{2}(t)) + H_2 *( \mu_{N_{2}}(t)+\bar{\nu}^{2}(t)) \right) (x_{5}(t), v_{5}(t)),
    \end{align*}
    with
\[
        \mu_{N_1} (t) = \frac{1}{2} \sum_{i\in\{3,4\}} \delta_{(x_i(t), v_i(t))},\quad
        \nu^{1} (t) = \delta_{(x_1(t), v_1(t))}, 
\]
\[
        \mu_{N_2} (t) = \frac{1}{2} \sum_{i\in\{2,5\}} \delta_{(x_i(t), v_i(t))},\quad
        \nu^{2} (t) = 0,\quad
        \bar{\nu}^{2} = \delta_{\left(x_1(t), v_1(t)\right)}.
\]
    Assume that for the reference vehicle $1$ the timer $\tau_1$ has reached the timer limit $T_1$. Then, it would consider to change lane if the conditions are fulfilled. In that case, the reference vehicle changes lane going to lane $2$ and the sets of indices of vehicles modify as
    \begin{align*}
        \iota(1) &= (1,2,\underline{1}_L,\underline{1}_F), \quad \underline{1}_L = (2,0,4), \ \underline{1}_F = (5,0,3) \\
        \iota(2) &= (2,2,\underline{2}_L,\underline{2}_F), \quad \underline{2}_L = (0,0,4), \ \underline{2}_F = (1,0,3) \\
        \iota(3) &= (3,1,\underline{3}_L,\underline{3}_F), \quad \underline{3}_L = (0,5,0), \ \underline{3}_F = (0,0,0)\\
        \iota(4) &= (4,1,\underline{4}_L,\underline{4}_F), \quad \underline{4}_L = (0,0,0), \ \underline{4}_F = (0,0,2) \\
        \iota(5) &= (5,2,\underline{5}_L,\underline{5}_F), \quad \underline{5}_L = (1,0,4), \ \underline{5}_F = (0,0,3).
    \end{align*}
\end{example}
}

\revision{The presence of both continuous dynamics of vehicles governed by system \eqref{ODEs} and discrete dynamics of vehicles caused by lane changing motivates us to consider \revision{a} finite-dimensional hybrid system. }

\subsection{\revision{Existence and uniqueness of solutions to the finite-dimensional hybrid system}}

For the definition of a hybrid system for multilane traffic, we need to introduce the following notation. \revision{We introduce $X \coloneqq \mathbb{R}\times \mathbb{R}_{\geq 0} \times [0, T_1)$, $\mathcal{I} \coloneqq \{1,\dots,M+N\}$, and $\ell \in\mathbb{R}^{M+N}$ such that $\ell_i=\pi_2(\iota(i))$, $i\in\mathcal{I}$, is the label of the lane of vehicle $i$. Then, we further define, 
\begin{equation}
\label{A_l}
A_\ell \coloneqq \left\{(x_i, v_i, \tau_i)_{i\in\mathcal{I} } \in X \colon \exists i \not = k \in \mathcal{I}, t\in [0, T], \text{ s.t. } \ell_{i} (t)= \ell_{k}(t) \wedge x_i(t) = i_k(t) \right\},
\end{equation}
the set of states such that two cars
are in same lane and position.}

\revision{
\begin{definition}
\label{Def_finite_HS_new}
A hybrid system modeling multilane traffic
is given by a $6$-tuple $\Sigma_1 = (\mathcal{L}, \mathcal{M}, U,\mathcal{U},g, S)$ where:
\vspace{0.5em}
\begin{enumerate}\setlength\itemsep{0.5em}
\item[(1)]  $ \mathcal{L} = \left\{\ell = (\ell_i), \ell_i \in J, i\in\mathcal{I} \right\}$ is a finite set of symbols that represent all possible lane labels of all vehicles;
\item[(2)]  $\mathcal{M} =\{\mathcal{M}_{\ell}\}_{\ell \in \mathcal{L}}$, 
where $\mathcal{M}_{\ell} = (X \setminus A_{\ell})^{M+N}$,
with $A_{\ell}$ defined as in \eqref{A_l}. 
\item[(3)]  $U= \left\{ U_{\ell}\right\}_{\ell \in \mathcal{L}}$,
$U_{\ell} = I^{M}$, where $ I \subset [-U_{\max}, U_{\max}]$ is compact with $U_{\max}>0$;
\item[(4)]  $\mathcal{U}= \left\{ \mathcal{U}_{\ell}\right\}_{\ell \in \mathcal{L}}$, $\mathcal{U}_{\ell} = \left\{ u \colon [0, T] \to U_{\ell} =I^{M}\right\}$;
\item[(5)]  $g = \{g_{\ell}\}_{\ell \in \mathcal{L}}$,
$g_{\ell} \colon \mathcal{M}_{\ell} \times \mathcal{U}_{\ell} \to \mathbb{R}^{3(M+N)}$, $g_{\ell_i}=(v_i, a_i, 1)$, where
$a_i=\dot{v}_i$ as defined in system \eqref{ODEs};
\item[(6)] $S \text{ is a subset of } LC(\Sigma_1)$, where 
$LC(\Sigma_1)$ is the set of states for which a lane-changing can occur, that is,
\begin{equation*} 
\begin{aligned}
&LC(\Sigma_1)= \left\{\left(\ell, (x_i, v_i,\tau_i),
\ell', ( x_i', v_i', \tau_i') \right)_{ i\in \mathcal{I}} \colon \exists k \in \mathcal{I}, t_{k} \in [0,T], \text{ s.t. } \forall i \not = k, \right. \\
 &\left.(x_i(t_{k}),v_i(t_{k}),\tau_i(t_{k}), \ell_i(t_{k})) = (x_i'(t_{k}),v_i'(t_{k}),\tau_i'(t_{k}), \ell_i'(t_{k})),\right.\\
&\left.\text{ and } (x_{k}(t_{k}), v_{k}(t_{k})) = (x_{k}'(t_{k}), v_{k}'(t_{k})), {\tau_{k}}'(t_{k}) = 0, \right.\\
&\left.\ell_{k}'(t_{k}) = (\ell_{k}(t_{k}) + 1) (1-\delta_{m}(\ell_{k}(t_{k}))) \text{ or } (\ell_{k}(t_{k}) - 1) (1-\delta_{1}(\ell_{k}(t_{k})))\right\}.
\end{aligned}
\end{equation*}
\end{enumerate}
\end{definition}
}

Before actually defining a trajectory of hybrid system $\Sigma_1$, it is necessary to define its hybrid state first. 

\begin{definition}
A hybrid state of the hybrid system $\Sigma_1$ is a \revision{$4$-tuple 
$(\ell, x, v, \tau)$, where $\ell\in\mathcal{L}$ is the location, $(x, v, \tau) \in \mathcal{M}_{\ell}$.} We denote by $\mathcal{HS}_1$ the set of all hybrid states of the hybrid system $\Sigma_1$. 
\end{definition}
Now we will define a trajectory of hybrid system $\Sigma_1$.

\revision{
\begin{definition}
\label{Def_trajectory_HS_1}
Let $(\ell_0, x_0, v_0, \tau_0) \in J^{M+N} \times \mathbb{R}^{M+N} \times {(\mathbb{R}_{\geq 0})}^{M+N} \times [0, \delta_{\tau})^{M+N}$ be given initial condition to the above hybrid system $\Sigma_1$. In addition, assume that the initial conditions $\tau_0$ satisfy condition \eqref{initial timer}. A trajectory of the hybrid system $\Sigma_1$ with initial condition $(\ell_0, x_0, v_0, \tau_0)$ is a map $\xi \colon [0, T] \to \mathcal{HS}_1$, $\xi(t) = (\ell(t),x(t), v(t), \tau(t))$ such that for $i \in \mathcal{I}, \text{ and } n \in \{ 1, \dots, N_{\tau} -1\}$, the following holds: 
\vspace{0.5em}
\begin{enumerate}\setlength\itemsep{0.5em}
\item[(1)] $(x_i(0), v_i(0), \tau_i(0) ) = (x_{i,0}, v_{i,0},  \tau_{i,0}) \in X$;
\item[(2)] $ \ell_i(t) = \ell_{i, 0} \in J$ for $t\in[0, \delta_{\tau} - \tau_{i,0})$, \\
$\ell_i(t)= \ell_{i,n} \in J$ for $t\in[n\delta_\tau - \tau_{i,0}, (n+1)\delta_\tau - \tau_{i,0})$,\\
$\ell_i (t)= \ell_{i,N_{\tau}} \in J$ for $t\in[N_{\tau}\delta_\tau - \tau_{i,0}, T]$;
\item[(3)] $\tau_i(n\delta_{\tau} - \tau_{i,0}) = 0$;
\item[(4)] $\lim\limits_{t \to (n\delta_{\tau} - \tau_{i,0})^{-}}x_{i}(t)= x_i(n\delta_{\tau} - \tau_{i,0})$;
\item[(5)] For almost every $t \in [0, T]$, with $u_i \colon [0, T] \to I$ a measurable control, \\
$\frac{\mathrm{d}}{\mathrm{d}t} (x_i, v_i, \tau_i) = g_{\ell_i(t)}(x_i(t), v_i(t), \tau_i(t), u_i(t))$.
\end{enumerate}
\end{definition}
}

We shall derive the existence and uniqueness of the trajectory of hybrid system $\Sigma_1$ in the sense of Definition \ref{Def_trajectory_HS_1}.
Let $\xi^{j} = (y^j, w^j)$ be the space-velocity of the autonomous vehicles in the $j$-the lane. Recall that we denote by $M_j$ and $N_j$ the number of autonomous vehicles and the number of human-driven vehicles in the $j$-the lane, respectively. 
Compare with Lemma $2.1$ in \cite{fornasier2014mean}, we have the following lemma. 
\begin{lemma}
\label{lemmasublinear}
Given two locally Lipschitz convolution kernels with sub-linear \\
growth $H_1 \colon \mathbb{R} \times \mathbb{R}^{+} \to \mathbb{R}$ and $H_2 \colon \mathbb{R} \times \mathbb{R} \to \mathbb{R}$, and given $\mu_n = \frac{1}{n} \sum\limits_{l=1}^{n} \delta_{(x_l, v_l)}$ an arbitrary atomic measure for $(x_l, v_l) \in \mathbb{R} \times \mathbb{R}^{+}$, with $n \in \mathbb{Z}^{+}$, we have, there exists a constant $C>0$ such that
$$
|H_1*_1 \mu_n(x,v)| \leq C(1+|(x,v)| + \frac{1}{n}\sum\limits_{l=1}^n |(x_l, 0)|)$$
and
$$|H_2*\mu_n(x,v)|\leq C(1+|(x,v)| + \frac{1}{n} \sum\limits_{l=1}^n|(x_l,v_l)|).
$$
\end{lemma}
\begin{proof}
This is a consequence of the sub-linear growth of $H_1$ and $H_2$. 
\end{proof}
As in \cite{fornasier2014mean}, motivated by the $1$-Wasserstein distance, we endow space $\mathbb{R}^{2n}$ for any $n \in \mathbb{Z}^{+}$ with the following norm: for any $(x,v) \in \mathbb{R}^{2n}$, 
$\|(x,v)\| \colon = \frac{1}{n} \sum\limits_{l=1}^{n} \left(|x_l|+|v_l|\right)$,
and the metric induced by the above norm $\|\cdot \|$. 

\begin{theorem} 
\label{thm_traj_HS_1}
Let $H_1 \colon \mathbb{R}\times \mathbb{R}^{+} \to \mathbb{R}$, $H_2\colon \mathbb{R}\times \mathbb{R} \to \mathbb{R}$ be locally Lipschitz convolution kernels with sub-linear growth. Then given an initial datum 
\revision{$\xi_0 = (\ell_0, x_0, v_0, \tau_0)$}, there exists a unique trajectory  \revision{$\xi(t) = (\ell(t), x(t), v(t), \tau(t))$} to the finite-dimensional hybrid system $\Sigma_1$ over the whole time interval $[0, T]$. Furthermore, both trajectories of the autonomous vehicles and the human-driven vehicles are Lipschitz continuous with respect to time over the time interval when there is no lane changing. 
\end{theorem}
\begin{proof}
Let \revision{$t_0 = \min\limits_{
i \in \mathcal{I}}\{T_1 - \tau_{i,0}\}$}. Note that there is no vehicle changing lane over the time interval $[0, t_0)$ in any lane. \revision{In particular, for each $j \in J$, define the set of labels of vehicles on lane $j$ as $\mathcal{I}_j = \{i \in \mathcal{I} \colon \pi_2(\iota(i)) \in J\}$. We recall that for $t \in [0, t_0)$, the dynamics of vehicle $i \in \mathcal{I}_j$ satisfy the system \eqref{ODEs}. 
For the sake of compact writing, we let $\xi^j(t)= (x^j(t), v^j(t)) \in (\mathbb{R}\times \mathbb{R}_{\geq 0})^{M_j+N_j}$ be the trajectories of all vehicles on lane $j$. Namely, $x^j(t)=(x_i(t))_{i\in\mathcal{I}_j}$ is the vector of the positions of vehicles on lane $j$, and similarly for $v^j(t)$. }Then, we re-write system \eqref{ODEs} in the following form 
\begin{equation}
\label{compact_form}
\dot{\xi^j}(t) = g^j(t, \xi^j(t)),
\end{equation}
where the right hand side is 
\revision{
\[g^j(t, \xi^j(t)) = (v^j(t), [\left(H_1*_1(\mu_{N_j} + \nu^{j})+H_2*(\mu_{N_j}+\nu^j)\right)(x_i, v_i)+u_i^j]_{i\in \mathcal{I}_j}),\]
where $u_i^j \equiv 0$ if $i \in \mathcal{I}_j \cap \{M+1, \dots, M+N\}$.}
Since $H_1$ and $H_2$ are locally Lipschitz with sub-linear growth, by Lemma \ref{lemmasublinear}, we obtain
\begin{align*}
& \|g^j(t, \xi^j(t))\| \leq \bar{C}\left(1+\|\xi^{j}(t)\|\right),
\end{align*}
where $\bar{C}>0$ is a constant depending on $C>0$, $U_{\max}>0$, but not depending on $M$ or $N$. 
Thus the right hand side of equation \eqref{compact_form} fulfills the sub-linear growth condition, by Theorem \ref{global_existence_uniqueness}, there exists a solution of system \eqref{compact_form} on the interval $[0, t_0)$ such that \revision{$\xi^j(0)= \xi^j_0=(x_0^j, v_0^j) \in (\mathbb{R} \times \mathbb{R}_{\geq 0})^{M_j + N_j}$}. 
Moreover, for any $t\in [0, t_0)$, 
\[\|\xi^j(t)\| \leq (\|\xi^j_0\| + \bar{C} t_0)e^{\bar{C}t_0}.\]
In addition, the trajectory of the  vehicles in lane $j$ is Lipschitz continuous in time over the interval $[0, t_0)$. That is, for any $\tau_1, \tau_2 \in [0, t_0)$, 
\begin{align*}
\|\xi^{j}(\tau_1) - \xi^{j}(\tau_2)\|& 
\leq \int_{\tau_1}^{\tau_2} \bar{C}(1+\|\xi^j(s)\|)\,\mathrm{d}s \leq \bar{C}(1+ (\|\xi^j_0\| + \bar{C} t_0)e^{\bar{C}t_0})|\tau_1 - \tau_2|.
\end{align*}
Now for $n \geq 1$, let \revision{$t_n = \min\limits_{
i=1, \dots, M+N}\{T_1 - \tau_{i}(t_{n-1})\}$}. Then over the time interval $[t_{n-1}, t_n)$, $n \geq 1$, there is no vehicle changing lane. Similarly, one can show that the trajectory of vehicles in lane $j$ is unique and is Lipschitz continuous in time over the time interval $[t_{n-1}, t_n)$. Since the number of vehicles $M+N$ is finite, one can repeat the above procedure for finitely many times to show that the trajectory of vehicles in lane $j$ is unique over the whole time interval $[0, T]$. 
\end{proof}
\section{The Mean-Field Limit to the Finite-Dimensional Hybrid System}
\label{sec_mean_field_limit_dimen_hybrid}
In this section, we consider \revision{again $M$ autonomous vehicles and $N$ human-driven vehicles. However, due to the large number of the latter compared to the former, we change perspective in the description of the behavior of regular vehicles considering the mean-field limit of their microscopic dynamics on each lane of an open stretch of road with $m$ lanes.} We again just add controls on the $M$ autonomous vehicles. \revision{We introduce the following notations $\mathcal{I}_M=\{1, \dots, M\}$ and for each $j \in J$, $\mathcal{I}_M^j = \{i \in \mathcal{I}_M \colon \pi_2(\iota(i))=j\}$.} It is possible to define a mean-field limit of system \eqref{ODEs} in the following sense: on lane $j \in J$, the population of vehicles can be represented by the vector of positions-velocities \revision{$(x^j, v^j) \in (\mathbb{R}\times \mathbb{R}_{\geq 0})^{M_j}$} of the autonomous vehicles, \revision{where from now on, we set $x^j=(x_i)_{i\in \mathcal{I}_M^j}$ and  $v^j=(v_i)_{i\in \mathcal{I}_M^j}$,} coupled with the compactly supported non-negative measure $\mu^j \in \mathcal{M}^{+}(\mathbb{R} \times \mathbb{R}_{\geq 0})$ of the human-driven vehicles in the position-velocity space. Then the mean-field limit will result in a coupled system of ODEs for \revision{$(x^j, v^j)$} with control and a PDE for $\mu^j$ without control. Furthermore, the lane changing of the human-driven vehicles would lead to a source term to the PDE for $\mu^j$. More specifically, the limit dynamics of vehicles on lane $j$ when there is no autonomous vehicles changing lane is 
\revision{
\begin{subequations}
\label{Mean-flied_limit}
\begin{align}
&\dot{x}_i = v_i; \label{AV_position}\\ 
&\dot{v}_i = \left(H_1*_1(\mu^{j} + \nu^{j})+H_2*(\mu^{j}+\nu^j)\right)(x_i, v_i)+u_i^j, \quad i\in \mathcal{I}_M^j; \label{AV_velocity}\\
&\partial_t \mu^j + v^j \partial_x \mu^j + \partial_{v} \left(\left(H_1*_1 (\mu^j+\nu^{j}) + H_2*(\mu^j+\nu^{j}) \right)\mu^j\right)
=S(\mu^{j-1},\mu^{j}, \mu^{j+1}). \label{RVPDE_HS_2}
\end{align}
\end{subequations}
}
where \revision{$u_i^j \colon [0, T] \to \mathbb{R}$} are measurable controls for \revision{$i \in \mathcal{I}_M^{j}$}, $H_1 \colon \mathbb{R}\times \mathbb{R}^{+} \to \mathbb{R}$ and $H_2\colon \mathbb{R}\times \mathbb{R} \to \mathbb{R}$ are locally Lipschitz convolution kernels with sub-linear growth satisfying equation \eqref{sublinear}, $\nu^{j}$ is as defined in \eqref{atomic_nu} and the source term $S(\mu^{j-1}, \mu^j, \mu^{j-1})$ is defined as
\begin{equation}
\label{source_term}
\begin{aligned}
S(\mu^{j-1},\mu^ {j}, \mu^{j+1}) =& \left(S^{j-1, j}(\mu^{j-1}, \mu^j) - S^{j, j-1}(\mu^{j-1}, \mu^j) \right)(1-\delta_{j,1})\\ 
&+\left(S^{j+1, j} (\mu^j, \mu^{j+1})- S^{j, j+1}(\mu^j, \mu^{j+1})\right)(1-\delta_{j,m}),
\end{aligned}
\end{equation}
with 
\revision{
\begin{equation}
\label{Macro-Macro}
S^{k,l}(\mu^k, \mu^l) = p([A^l + \Delta]_{+}, [A^l - A^k - \Delta]_{+}) \mu^k, k,l \in \{j-1, j, j+1\}
\end{equation}
}
$\text{and } \ k = l +1 \ \text { or } \ k=l -1.$ 
Here $p \colon \mathbb{R}\times \mathbb{R} \to [0, 1]$ is increasing and is the probability of the large population of human-driven vehicles performing lane changing from $k$ lane to $l$ lane. In addition, if $a, b\leq 0$, then $p(a, b) =0$. \revision{For consistency, we need to assume that the dimension of $p$ to be $\left[\mathrm{sec}\right]^{-1}$}. This modeling choice is similar to~\cite{IllnerKlarMaterne,HertyIllnerKlarPanferov}. In addition, $A^{l} = H_1*_1 (\mu^{l}+\nu^{l}) + H_2*(\mu^{l}+\nu^{l})$ is the average acceleration of vehicles on lane $l$. Equation \eqref{Macro-Macro} can be interpreted as the following: Let $\Delta>0$ be fixed. A large population of human-driven vehicles on lane $k$ will perform lane changing to lane $l$ with probability $p \in [0,1]$ if the following condition occur: $A^l > A^k + \Delta$. 

Furthermore, system \eqref{Mean-flied_limit} implies that the acceleration of autonomous vehicle \revision{$i \in \mathcal{I}_M^j$} is, 
\revision{
\begin{equation}
\label{eqn: acc_av}
a_i^j = \left(H_1*_1(\mu^{j} + \nu^{j})+H_2*(\mu^{j}+\nu^j)\right)(x_i, v_i)+u_i^j.
\end{equation}}
The $i$-th autonomous vehicle on lane $j$ will perform lane changing to $j'=(j-1)(1-\delta_1(j))$ or $j'=(j+1)(1-\delta_m(j))$ lane if the following condition occur:
$ A^{j'} \geq a_i^j + \Delta$.

We again assign each autonomous vehicle a timer over the whole time interval $[0,T]$ such that there are no two autonomous vehicles changing lane at the same time. We define the timer \revision{$\tau_i$} for autonomous vehicle \revision{$i \in \mathcal{I}_M$} and the timer limit $T_1$ as before. 

The continuous dynamics of vehicles governed by system \eqref{Mean-flied_limit} and the discrete lane changing dynamics of the autonomous vehicles lead us to consider the following hybrid system. 
\revision{
\begin{definition}
A hybrid ODE-PDE system is a $6$-tuple $\Sigma_2 = (\mathcal{L}, \mathcal{M}, U, \mathcal{U}, g, S)$ where
\begin{itemize}
\item[(1)]$
\mathcal{L} = \left\{\ell = (\ell_i)_{i \in \mathcal{I}_M}, \ell_i \in J\right\} = J^M \text{ is the set of locations} $;
\item[(2)]$\mathcal{M} =\{\mathcal{M}_{\ell}\}_{\ell \in \mathcal{L}}, \text{ where } \mathcal{M}_{\ell} = X^M \setminus A_{\ell} \times \left(\mathcal{M}^{+}(\mathbb{R}^2)\right)^m$ and
\begin{equation*}
\begin{aligned}
    A_\ell = \left\{\left(x_i, v_i, \tau_i \right)_{i \in \mathcal{I}_M} \in X \colon \right. & \exists i_1 \not = i_2 \in \mathcal{I}_M, t \in [0, T], \text{ s.t. }\\ 
    & \left. \ell_{i_1}(t) = \ell_{i_2}(t) \wedge x_{i_1}(t) = x_{i_2}(t) \right\};
\end{aligned}
\end{equation*}
\item[(3)] $U= \left\{ U_{\ell}\right\}_{\ell \in \mathcal{L}}, U_{\ell} = I^{M}, \text{ where } I \subset \mathbb{R} \text{ is compact}$;
\item[(4)] $\mathcal{U}= \left\{ \mathcal{U}_{\ell}\right\}_{\ell \in \mathcal{L}},
\mathcal{U}_{\ell} = \left\{ u \colon [0, T] \to U_{\ell} =I^{M}\right\}$; 
\item[(5)] $g = \{g_{\ell}\}_{\ell \in \mathcal{L}}$,
$g_{\ell} \colon \mathcal{M}_{\ell} \times \mathcal{U}_{\ell} \to (\mathbb{R}^3)^M$ with
\[g_{\ell}((x_i, v_i, \tau_i, u_i^{\ell_i}, \mu^{\ell_i})_{i \in \mathcal{I}_M}) = (v_i, a_i^{\ell_i}, 1)_{i\in \mathcal{I}_M},\]
$\text{ where } a_i^{\ell_{i}} \text{ is defined as in equation \eqref{eqn: acc_av}}$; 
\item[(6)]$S \text{ is a subset of } LC(\Sigma_2), \text{ where }$
\begin{equation*}
\begin{aligned}
& LC(\Sigma_2) = \Big\{\big(\ell, (x_i, v_i, \tau_i, \mu^{\ell_i}), \ell', (x_i', v_i', \tau_i', \mu^{\ell'_{i}})\big)_{i \in \mathcal{I}_M} \colon
\exists k \in \mathcal{I}_M, t_k \in [0, T], \\
& \text{ s.t. } \forall i \not = k, (x_i(t_k), v_i(t_k), \tau_i(t_k), \ell_i(t_k)) =(x_i'(t_k), v_i'(t_k), \tau_i'(t_k), \ell_i'(t_k)),\\ &  \text{ and } x_k(t_k) = x_k'(t_k), v_k(t_k) = v_k'(t_k), \tau_k'(t_k)=0,\\
&\ell_{k}'(t_{k}) = (\ell_{k}(t_{k}) + 1) (1-\delta_{m}(\ell_{k}(t_{k}))) \text{ or } (\ell_{k}(t_{k}) - 1) (1-\delta_{1}(\ell_{k}(t_{k})))
\Big\}.
\end{aligned}
\end{equation*}
\end{itemize}
\end{definition}
}
Now we will define the hybrid state of the hybrid system $\Sigma_2$.
\revision{
\begin{definition}
A hybrid state of the hybrid system $\Sigma_2$ is a $5$-tuple $(\ell, x, v, \tau, \mu)$, where $\ell$ is the location, $(x, v, \tau, \mu) \in \mathcal{M}_{\ell}$. We denote by $\mathcal{HS}_2$ the set of all hybrid states of the hybrid system $\Sigma_2$. 
\end{definition}
}
Next we will give the definition of the trajectory of the hybrid system $\Sigma_2$.
\revision{
\begin{definition}
\label{Def_trajectory_HS_2}
A trajectory of the hybrid system $\Sigma_2$ with initial condition \\
$(\ell_0, x_0, v_0, \tau_0, \mu_0) \in \mathcal{L} \times X^M \times \left(\mathcal{M}^{+}(\mathbb{R}^2)\right)^m (\text{ if } i_1 \not = i_2 \in \mathcal{I}_M, \text{ then } \tau_{i_1,0} \not = \tau_{i_2, 0})$ is a map $\xi \colon [0, T] \to \mathcal{HS}_2$,
$\xi(t) = (\ell(t), x(t), v(t), \tau(t), \mu(t))$ such that for $i\in \mathcal{I}_M$ and $n=1, \dots, N_{\tau}-1$, the following holds:
\begin{itemize}
\item[(1)] $(x_i(0), v_i(0), \tau_i(0)) = (x_{i,0}, v_{i,0}, \tau_{i,0}) \in X$;
\item[(2)] $\text{For }t \in [0, T_1- \tau_{i,0}), \ell_i(t) = \ell_{i,0} \in J$, 
$\ell_i(\cdot) \text{ is constant in } [nT_1 - \tau_{i,0}, (n+1)T_1 - \tau_{i,0}), \text{ and is equal to } \ell_{i,n} \in J$;
\item[(3)] $\tau_i(nT_1 - \tau_{i,0}) = 0$;
\item[(4)] $\lim\limits_{t \to \left(nT_1 - \tau_{i,0} \right)^{-}}x_i(t) \text{ exists and is equal to } x_i(nT_1 - \tau_{i,0})$;
\item[(5)]$\text{For every }\varphi \in C_c^{\infty}(\mathbb{R} \times \mathbb{R}_{\geq 0}),\text{ and for all }t\in [0, T], \mu^{\ell_i(t)} \text{ satisfies}$\\
$ \supp{\mu^{\ell_i(t)}} \subset B(0, R) \text{ for some } R>0, \text{ and for almost every } t\in [0, T]$, \\
$\frac{\mathrm{d}}{\mathrm{d}t} \int_{\mathbb{R}\times \mathbb{R}_{\geq 0}} \varphi(x, v)\,\mathrm{d}\mu^{\ell_i(t)}(t)(x,v)=\\
= \int_{\mathbb{R}\times \mathbb{R}_{\geq 0}} \varphi(x,v)\,\mathrm{d}S(\mu^{\ell_i(t)-1}, \mu^{\ell_i(t)}, \mu^{\ell_i(t)+1})(t)(x,v)$\\ 
$+\int_{\mathbb{R}\times \mathbb{R}_{\geq 0}} \left(\nabla \varphi(x,v) \cdot \omega_{H_1, H_2, \mu^{\ell_i(t)}, x^{\ell_i(t)}, v^{\ell_i(t)}}(t, x, v)\right)\,\mathrm{d}\mu^{\ell_i(t)}(t)(x,v)$,\\
$\text{where } \omega_{H_1, H_2, \mu^{\ell_i(t)}, x^{\ell_i(t)}, v^{\ell_i(t)}}(t, x, v) \coloneqq$ \\
$= \left(v, \left(H_1*_1(\mu^{\ell_i(t)}(t)+\nu^{\ell_i(t)}(t)) + H_2*(\mu^{\ell_i(t)}(t)+\nu^{\ell_i(t)}(t))\right)(x,v)\right).$
\item[(6)]$ \text{For almost every } t \in [0, T], \text{ with } u_i^{\ell_i} \colon [0, T] \to I $ \text{ a measurable control } \\
$\frac{\mathrm{d}}{\mathrm{d}t} (x_i(t), v_i(t), \tau_i(t)) = g_{\ell_{i}(t)}(x_i(t), v_i(t), \tau_i(t), u_i^{\ell_i}(t), \mu^{\ell_i(t)}(t))$.
\end{itemize}
\end{definition}
}
Before actually proving the existence of trajectories of the hybrid system $\Sigma_2$ as in Definition \ref{Def_trajectory_HS_2}, it will be convenient to address the stability of the hybrid system $\Sigma_2$ with respect to the initial data first. 

Let \revision{$t_0^1 = \min \limits_{i\in \mathcal{I}_M}\left\{T_1 - \tau_{i,0}\right\}$.} Then there is no autonomous vehicle changing lane over the time interval $[0, t_0^1)$ on any lane. As in Theorem \ref{thm_traj_HS_1}, it is enough to show the stability of the hybrid system $\Sigma_2$ with respect to the initial data over the time interval $[0, t_0^1)$. \revision{In particular, for $t \in [0, t_0^1)$, the dynamics of autonomous vehicle $i \in \mathcal{I}_M$ and the human-driven vehicles in lane $\ell_i$ satisfy system \eqref{Mean-flied_limit} with the following initial conditions: 
$(x_i(0), v_i(0), \mu^{\ell_i}(0))= (x_{i,0}, v_{i,0}, \mu_0^{\ell_i}) \in \mathbb{R} \times \mathbb{R}_{\geq 0} \times \mathcal{M}^{+}(\mathbb{R}\times \mathbb{R}_{\geq 0})$.}
Furthermore, we endow space $\mathcal{X}_n \colon \mathbb{R}^{2n} \times \mathcal{M}^{+}(\mathbb{R} \times \mathbb{R}_{\geq 0})$ for any $n \in \mathbb{Z}^{+}$ with the following metric: for any $(x_1, v_1, \mu_1), (x_2, v_2, \mu_2)\in \mathcal{X}_n$, 
\[\|(x_1, v_1, \mu_1) - (x_2, v_2, \mu_2)\|_{\mathcal{X}_n} \colon = \frac{1}{n} \sum\limits_{k=1}^{n} \left(|x_{i,1} - x_{i, 2}|+|v_{i,1} - v_{i,2}|\right) + W_1^{1,1}(\mu_1, \mu_2), \]
where $W_1^{1,1}$ is the generalized Wasserstein distance in $ \mathcal{M}^{+}(\mathbb{R} \times \mathbb{R}_{\geq 0})$. 

\begin{lemma}
For $j \in J$, and $q \in \{1, 2\}$, let $\mu^{j,q}$ be two solutions to system \eqref{RVPDE_HS_2} over the time interval $[0, t_0^1)$ with two different initial data $\mu_0^{j,q} \in \mathcal{M}^{+}(\mathbb{R} \times \mathbb{R}_{\geq 0})$. Then there exists $\bar{C}>0$ such that, 
\begin{align}
\label{eqn_gdw_inequality_new}
& W_1^{1,1}(\mu^{j,1}(t), \mu^{j,2}(t)) \leq \bar{C} \left(W_1^{1,1}(\mu_0^{j,1}, \mu_0^{j,2})+ \right. \\ \nonumber 
&\left.+\int_0^t \|(x^{j,1}(s), v^{j,1}(s), \mu^{j, 1}(s)) -(x^{j,2}(s), v^{j,2}(s), \mu^{j, 2}(s)) \|_{\mathcal{X}_{M_j}}\,\mathrm{d}s\right).
\end{align}
\end{lemma}
\begin{proof}
Let $\mu^{j,q}$ be two solutions to system \eqref{RVPDE_HS_2} over the time interval $[0, t_0^1)$ with two different initial data $\mu_0^{j,q}$, $q=1, 2$. Let $t \in [0, t_0^1)$ be fixed and let $\Delta t = \frac{t_0^1}{2^k}$ for a fixed $k \in \mathbb{N}^{+}$. Decompose the time interval $[0, t_0^1)$ into $[0, \Delta t]$, $[\Delta t, 2 \Delta t], \dots, [(2^k-1)\Delta t, 2^k \Delta t)$. Let $n$ be the maximum integer such that $t - n\Delta t \geq 0$, then $t \in [n \Delta t, (n+1)\Delta t)$. By section \ref{sec_PDE_source}, we have, $\mu^{j, q}(t) = \lim \limits_{k \to \infty} \mu_k^{j,q}(t)$, where $q=1, 2$ and $\mu_k^{j,q}$ is defined as following:
\begin{align*}
& \mu_k^{j,q}(0) = \mu_0^{j,q},\\
& \mu_k^{j,q}((n+1)\Delta t) = \mathcal{T}_{\Delta t}^{\mu_k^{j,q}(n\Delta t), \nu^{j,q}(n\Delta t)} \# \mu_k^{j,q}(n \Delta t) + \Delta t S(\mu_{k}^{j,q}(n\Delta t)),\\
& \mu_k^{j,q}(t) = \mathcal{T}_{\tau}^{\mu_k^{j,q}(n \Delta t), \nu^{j,q}(n\Delta t)}\#\mu_{k}^{j,q}(n\Delta t) + \tau S(\mu_{k}^{j,q}(n \Delta t)),
\end{align*}
where $\tau = t - n\Delta t$ and $\nu^{j,q}(n\Delta t) = \frac{1}{M_j}\sum\limits_{i=1}^{M_j}\delta_{\left(x_i^{j,q}(n\Delta t), v_{i}^{j,q}(n\Delta t)\right)}$, with \\$(x_i^{j,q}(n\Delta t), v_i^{j,q}(n\Delta t))$ being the vector of position-velocity of the $i$-th autonomous vehicle on lane $j$ at time $n\Delta t$ when the initial data to system \eqref{RVPDE_HS_2} is given by $\mu_0^{j, q}$. 
Note that 
\begin{align*}
& W_1^{1,1}(\mu_k^{j,1}(t), \mu_k^{j,2}(t)) \leq W_{1}^{1,1}\left(\tau S(\mu_{k}^{j,1}(n \Delta t)), \tau S(\mu_{k}^{j,2}(n \Delta t)) \right) \\
&+ W_1^{1,1}\left( \mathcal{T}_{\tau}^{\mu_k^{j,1}(n \Delta t), \nu^{j,1}(n \Delta t)}\#\mu_{k}^{j,1}(n\Delta t), \mathcal{T}_{\tau}^{\mu_k^{j,1}(n \Delta t), \nu^{j,1}(n \Delta t)}\#\mu_{k}^{j,2}(n\Delta t)\right)\\
&+ W_1^{1,1} \left( \mathcal{T}_{\tau}^{\mu_k^{j,1}(n \Delta t), \nu^{j,1}(n \Delta t)}\#\mu_{k}^{j,2}(n\Delta t), \mathcal{T}_{\tau}^{\mu_k^{j,2}(n \Delta t), \nu^{j,2}(n \Delta t)}\#\mu_{k}^{j,2}(n\Delta t)\right),
\end{align*}
where the last inequality is due to Proposition \ref{pro_gwd_property}. 
By the properties of the source term $S$, $(S_2)$, and of the generalized Wasserstein distance $W_1^{1,1}$, Proposition \ref{pro_gwd_property}, there exists some constant $L_S$ such that 
\begin{align*}
W_{1}^{1,1}\left(\tau S(\mu_{k}^{j,1}(n \Delta t)), \tau S(\mu_{k}^{j,2}(n \Delta t)) \right) \leq \tau L_SW_1^{1,1}(\mu_k^{j,1}(n\Delta t), \mu_k^{j,2}(n\Delta t)).
\end{align*}
Since the flow map $\mathcal{T}_{\tau}^{\mu_k^{j,1}(n\Delta t), \nu^{j,1}(n \Delta t)}$ is Lipschitz, by Lemma \ref{lm_66_gwd}, there exists some constant $L_1$, such that, 
\begin{align*}
& W_1^{1,1}\left( \mathcal{T}_{\tau}^{\mu_k^{j,1}(n \Delta t), \nu^{j,1}(n \Delta t)}\#\mu_{k}^{j,1}(n\Delta t), \mathcal{T}_{\tau}^{\mu_k^{j,1}(n \Delta t), \nu^{j,1}(n \Delta t)}\#\mu_{k}^{j,2}(n\Delta t)\right)\\
\leq & 
L_1 W_1^{1,1}(\mu_k^{j,1}(n\Delta t), \mu_k^{j,2}(n \Delta t)). 
\end{align*}
Since the flow maps $\mathcal{T}_{\tau}^{\mu_k^{j,1}(n \Delta t), \nu^{j,1}(n\Delta t)}$ and $\mathcal{T}_{\tau}^{\mu_k^{j,2}(n \Delta t), \nu^{j,2}(n\Delta t)}$ are bounded and Borel measurable, by Lemma \ref{lm_66_gwd}, equation \eqref{eqn_flow_bounded_Lip} and Lemma \ref{lm_67_gwd}, there exist $L_{\mathcal{T}}, \rho,L_{*}>0$, such that
\begin{align*} 
& W_1^{1,1}\left(\mathcal{T}_{\tau}^{\mu_k^{j,1}(n \Delta t), \nu^{j,1}(n \Delta t)}\#\mu_{k}^{j,2}(n\Delta t), \mathcal{T}_{\tau}^{\mu_k^{j,2}(n \Delta t), \nu^{j,2}(n \Delta t)}\#\mu_{k}^{j,2}(n\Delta t)\right)\\ \leq & \left\|\mathcal{T}_{\tau}^{\mu_k^{j,1}(n \Delta t), \nu^{j,1}(n \Delta t)}- \mathcal{T}_{\tau}^{\mu_k^{j,2}(n \Delta t), \nu^{j,2}(n \Delta t)}\right\|_{L^{\infty}(B(0, R))}\\
\leq & L_{*} \int_{n\Delta t}^{t} e^{L_{\mathcal{T}}(s-t)} \left[\left(\frac{1}{M_j} \sum\limits_{i=1}^{M_j} (|x_i^{j,1}(s) - x_i^{j,2}(s)|+|v_{i}^{j,1}(s)-v_{i}^{j,2}(s)|)\right)\right.\\
\quad\quad\quad & \left. W_1^{1,1}(\mu_k^{j,1}(s), \mu_k^{j,2}(s))\right]\,\mathrm{d}s. 
\end{align*}
Therefore, 
\begin{equation}
\label{eqn_unique_1}
\begin{aligned}
& W_1^{1,1}(\mu_k^{j,1}(t), \mu_k^{j,2}(t)) \leq \left(\tau L_S + L_1\right)W_1^{1,1}(\mu_k^{j,1}(n\Delta t), \mu_k^{j,2}(n\Delta t))\\ 
& +L_{*} \int_{n\Delta t}^{t} e^{L_{\mathcal{T}}(s-t)} \left[\left(\frac{1}{M_j} \sum\limits_{i=1}^{M_j} \left(|x_i^{j,1}(s) - x_i^{j,2}(s)|+|v_{i}^{j,1}(s)-v_{i}^{j,2}(s)|\right)\right) \right.\\ 
& \left.+ W_1^{1,1}(\mu_k^{j,1}(s), \mu_k^{j,2}(s))\right]\,\mathrm{d}s. 
\end{aligned}
\end{equation}

Similarly, there exists $L_2>0$, such that 
\begin{align}
\label{eqn_unique_2}
& W_1^{1,1}(\mu_k^{j, 1}(n\Delta t), \mu_{k}^{j, 2}(n\Delta t)) \\ \nonumber 
\leq & (L_2+\Delta t L_S) W_1^{1,1}(\mu_k^{j,1}((n-1)\Delta t), \mu_k^{j,2}((n-1)\Delta t))\\ \nonumber 
& + L_{*} \int_{(n-1)\Delta t}^{n \Delta t } e^{L_{\mathcal{T}}(s-t)} \left[\left(\frac{1}{M_j} \sum\limits_{i=1}^{M_j} \left(|x_i^{j,1}(s) - x_i^{j,2}(s)|+|v_i^{j,1}(s)-v_i^{j,2}(s)|\right)\right) \right.\\ \nonumber 
& \left.+ W_1^{1,1}(\mu_k^{j,1}(s) - \mu_k^{j,2}(s))\right]\,\mathrm{d}s. 
\end{align}

Combine with equations \eqref{eqn_unique_1} and \eqref{eqn_unique_2}, and the definition of norm $\|\cdot \|_{\mathcal{X}_{M_j}}$, we obtain 
there exists $C_0$ such that 
\begin{align*}
& W_1^{1,1}(\mu_k^{j,1}(t), \mu_k^{j,2}(t)) \leq C_0 \left(W_1^{1,1}(\mu_0^{j,1}, \mu_0^{j,2}) \right. \\
&\left.+\int_0^t \|(x^{j,1}(s), v^{j,1}(s), \mu_k^{j, 1}(s)) -(x^{j,2}(s), v^{j,2}(s), \mu_k^{j, 2}(s)) \|_{\mathcal{X}_{M_j}}\,\mathrm{d}s\right).
\end{align*}
Take $k \to \infty$ and consider the definition of $\mu^{j,p}$, $p=1,2$, we have, there exists $\bar{C}$ such that inequality \eqref{eqn_gdw_inequality_new} is true. 
\end{proof}


\begin{theorem}
\label{thm_uniqueness_traj_HS_2}
Let $(x^{j,i}, v^{j,i})$, $i=1,2$, be two solutions of system \eqref{AV_position}-\eqref{AV_velocity} relative to given respective initial data $(x_{0}^{j,i}, v_{0}^{j,i}) \in \mathbb{R} \times \mathbb{R}_{\geq 0}$ and let $\mu^{j,i}$, $i=1, 2$, be two solutions of system \eqref{RVPDE_HS_2} relative to given respective initial data $\mu^{j,i}_{0}$, over the time interval $[0, t_0^1)$. Then there exists a constant $C>0$ such that 
\begin{align*}
& \left\|\left(x^{j,1}(t), v^{j,1}(t), \mu^{j,1}(t)\right)-\left(y^{j,2}(t), w^{j,2}(t), \mu^{j,2}(t)\right)\right\|_{\mathcal{X}_{M_j}} \\
\leq & C \left\|\left(x^{j,1}_0, v^{j,1}_0, \mu^{j,1}_0\right)-\left(x^{j,2}_0, v^{j,2}_0, \mu^{j,2}_0\right)\right\|_{\mathcal{X}_{M_j}}.
\end{align*}
\end{theorem}
\begin{remark}
Theorem \ref{thm_uniqueness_traj_HS_2} implies that the trajectory of hybrid system $\Sigma_2$, if it exists, is uniquely determined by the initial conditions. 
\end{remark}
\begin{proof}
By integration we have, for $t \in [0, t_0^1)$, 
\[x_k^{j, i}(t) = \int_0^t v_k^{j,i}(s) \,\mathrm{d}s + v_{k,0}^{j,i}, \quad i=1, 2, \quad k\in\mathcal{I}_M.\]
Thus 
\begin{equation}
\label{eqn48}
|x_k^{j,1}(t) - x_k^{j,2}(t)| \leq |x_{k,0}^{j,1} - x_{k,0}^{j,2}| + \int_0^t |v_k^{j,1}(s) - v_k^{j,2}(s)|\,\mathrm{d}s.
\end{equation}
In addition, by Lemma \ref{lm_67_gwd}, there exists a constant $L_R$, such that 
\begin{equation}
\label{eqn_49} 
\begin{aligned}
&|v_k^{j,1}(t) - v_k^{j,2}(t)| 
\leq |v_{k,0}^{j,1}- v_{k,0}^{j,2}| \\
&+ L_R \int_0^t \left(\frac{1}{M_j} \sum_{k=1}^{M_j}\left( |x_k^{j,1}(s) - x_{k}^{j,2}(s)|+|v_k^{j,1}(s) - v_{k}^{j,2}(s)|\right) \right. \\
& \quad \quad \quad \quad \left.+W_1^{1,1} (\mu^{j,1}(s), \mu^{j,2}(s))\right)\,\mathrm{d}s.
\end{aligned}
\end{equation}
Combine with equations \eqref{eqn_gdw_inequality_new} \eqref{eqn48}, \eqref{eqn_49}, and the definition of the norm $\|\cdot\|_{\mathcal{X}_{M_j}}$, we have, there exists a constant $C$, s.t.,
\begin{align*}
&\|(x^{j,1}(t), v^{j,1}(t), \mu^{j,1}(t))-(x^{j,2}(t), v^{j,2}(t), \mu^{j,2}(t))\|_{\mathcal{X}_{M_j}}\\
\leq & C\left(\|(x_0^{j,1}, v_0^{j,1}, \mu_0^{j,1})-(x_0^{j,2}, v_0^{j,2}, \mu_0^{j,2}))\|_{\mathcal{X}_{M_j}} \right.\\
& \left.+\int_0^t \|(x^{j,1}(s), v^{j, 1}(s), \mu^{j,1}(s))-(x^{j,2}(s), v^{j, 2}(s), \mu^{j,2}(s))\|_{\mathcal{X}_{M_j}}\,\mathrm{d}s\right). 
\end{align*}
One can conclude the stability estimate by applying Gronwall's inequality. 
\end{proof}

We shall now derive the existence of the trajectory of the hybrid system $\Sigma_2$. It is enough to show that the trajectories of the vehicles exist over the time interval $[0, t_0^1)$.

\begin{theorem}
\label{thm_main}
On lane $j\in J$, let $(x^j_{k,0}, v^j_{k,0}) \in \mathbb{R} \times \mathbb{R}_{\geq 0}$, $k\in \mathcal{I}_M$, $\mu^j_0 \in \mathcal{M}^{+}(\mathbb{R} \times \mathbb{R}_{\geq 0})$ and $u_{*} \in L^1([0,T], \mathcal{U})$ be given. In addition, assume that $\mu_0^j$ is of bounded support in $B(0, R)$ for $R>0$. Then the trajectories of the vehicles exist on lane $j$ over the time interval $[0, t_0^1)$.
\end{theorem}
\begin{proof}
We will first construct a sequence of atomic measures to approximate the measure $\mu_0^j$ in generalized Wasserstein distance. For every $N_j \in \mathbb{N}^{+}$, consider the atomic measure
\begin{equation}
\label{eqn_decomposition}
\mu_{0}^{N_j} = \sum\limits_{i=1}^{N_j} \revision{m_h} \delta_{\left(x_{i,0}^{N_j}, v_{i,0}^{N_j}\right)},
\end{equation} 
\revision{with $m_h = \tfrac{\sum\limits_{j=1}^m \|\mu_0^j\|}{\sum\limits_{j=1}^m N_j}$}, such that $\lim\limits_{N_j\to \infty} W_1^{1,1}(\mu_{0}^{N_j}, \mu_0^j) =0$. 
Here we call \revision{$m_h$} the average mass of the human-driven vehicle. 

In addition, fix a weakly convergent sequence $(u_{N_j})_{N_j \in \mathbb{N}}$ in $L^1([0, T], \mathcal{U})$ of control functions such that $u_{N_j}\rightharpoonup u_{*}$ in $L^1([0, T], \mathcal{U})$ as $N_j \to \infty$. By Theorem \ref{thm_traj_HS_1}, for each initial datum $\xi_0^{N_j} = (x_{N_j}^0, v_{N_j}^0, x_0^{N_j}, v_{0}^{N_j})\in (\mathbb{R}\times\mathbb{R}_{\geq 0})^{M_j}\times (\mathbb{R}\times\mathbb{R}_{\geq 0})^{N_j}$ depending on $N_j$, there exists a unique trajectory of the hybrid system $\Sigma_1$ with control $u_{N_j}$ over the time interval $[0, t_0^1)$.

Denote the trajectories of the vehicles on lane $j$ over the time interval $[0, t_0^1]$ with $\xi^{N_j}(t) = (x_{N_j}(t), v_{N_j}(t), \mu^{N_{j}}(t)) \in \mathcal{X}_{M_j}$. Here we identify $\mu^{N_j}(t) \in \mathcal{M}^{+}(\mathbb{R} \times \mathbb{R}_{\geq 0})$ the atomic measure of the human-driven vehicles with position-velocity $(x^{N_j}(t), v^{N_j}(t))$. 

By Theorem \ref{thm_traj_HS_1}, the trajectories of the vehicles are Lipschitz continuous with respect to time over the time interval when there is no lane changing. Furthermore, note that the average mass of a human-driven vehicle $\revision{m_h} \to 0$ as $N_j \to \infty$. Thus there exists $L>0$, such that for any $\varepsilon>0$, there exists $\tilde{N}_j>0$, such that whenever $N_j \geq \tilde{N}_j$, 
$
\|\xi^{N_j}(t) - \xi^{N_j}(s)\|_{\mathcal{X}_{M_j}} \leq L|t-s| + \min\{\varepsilon, |s-t|\}.
$
By Theorem \ref{revised_AA}, there exists a sub-sequence, again denoted by $\xi^{N_j}(\cdot) = (x_{N_j}(\cdot), v_{N_j}(\cdot), \mu^{N_j}(\cdot))$ converging uniformly to a limit $\xi^{*,j}(\cdot) = (x_{*}^j(\cdot), v_{*}^j(\cdot), \mu_{*}^j(\cdot))$. 
We will first verify that $(x_{*}^j(\cdot), v^j_{*}(\cdot))$ is a solution of system \eqref{AV_position}-\eqref{AV_velocity} for $\mu^j = \mu^j_{*}$ and $u^j = u_{N_j}$. 

Note that $\xi^{N_j} \darrow \xi^{*,j}$ implies that 
\begin{align*}
& (x_{N_j}(t), v_{N_j}(t)) \darrow (x_*^j(t), v_{*}^j(t)) \text{ in } [0, t_0^1); \\
& (\dot{x}_{N_j}(t), \dot{v}_{N_j}(t)) \rightharpoonup (\dot{x}_{*}^j(t), \dot{v}_{*}^j(t)) \text{ in } L^1([0, t_0^1), \mathbb{R}\times \mathbb{R}^{+});\\
& \lim \limits_{N_j \to \infty} W_1^{1,1}(\mu^{N_j}(t), \mu_{*}^j(t)) = 0.
\end{align*}
In particular, 
$\dot{x}_{k,*}^j(t) = v_{k, *}^j(t), \text{ for all } k=1, \dots, M_j$.
Furthermore, let us denote now 
\[\nu_{N_j} = \frac{1}{M_j} \sum \limits_{k=1}^{M_j} \delta_{(x_{k,N_j}(t), v_{k,N_j}(t))} \text{ and } \nu_{*}^{j} = \frac{1}{M_j} \sum \limits_{k=1}^{M_j} \delta_{(x_{k, *}^j(t), v_{k,*}^j(t))}. \]
By the uniform convergence of the trajectories and Lemma \ref{lm_67}, we have, as $N_j \to +\infty$,
$ W_1(\nu_{N_j}(t), \nu_{*}^j(t)) \to 0 $.
In addition, by the sublinear growth of $H_1$ and $H_2$, we have, as $N \to \infty$, 
\begin{align*}
&(H_1*_1(\mu^{N_j}+\nu_{N_j})+H_2*(\mu^{N_j}+\nu_{N^j}))(x_{k,N_j}(t), v_{k,N_j}(t)) \\
\darrow & (H_1*_1(\mu_{*}^j + \nu_{*}^j)+H_2*(\mu_{*}^j + \nu_{*}^j))(x_{k,*}^j(t), v_{k,*}^j(t)).
\end{align*}
By the weak convergence of $u_{N_j} \text{ to } u_{*}$ and of $\dot{v}_{N_j} \text{ to } \dot{v}_{*}^j$, for every $\tau \in [0, t_0^1]$, \\
$
\int_0^{\tau} \dot{v}_{k, *}^j(t)\,dt =
\int_0^{\tau} \left((H_1*_1(\mu_*^j + \nu_{*}^j)+H_2*(\mu_{*}^j + \nu_{*}^j))(x_{k,*}^j(t), v_{k,*}^j(t))+u_{k,*}^j(t)\right)\,\mathrm{d}t$.

Now we will verify that $\mu_{*}^j$ is a solution to system \eqref{RVPDE_HS_2} for $\nu^j = \nu_{*}^j$. 
For any time $t \in [0, t_0^1]$, let $N^1_j$ be the number of human-driven vehicles that still stay on lane $j$ and let $(x_i^{N^1_j}(t), v_i^{N^1_j}(t))$ be the location-velocity of the $i$-th human-driven vehicle that does not perform lane changing on lane $j$. Then we can track the position of those human-driven vehicles by an atomic measure
$$\mu^{N_j^1}(t) = \sum\limits_{i=1}^{N^1_j} \revision{m_h} \delta_{\left(x_i^{N^1_j}(t), v_{i}^{N^1_j}(t)\right)}.$$

For all $\varphi \in C_c^{\infty}(\mathbb{R} \times \mathbb{R}^{+})$, consider the following differentiation 
\begin{align*}
& \frac{\mathrm{d}}{\mathrm{d}t} \langle \varphi, \mu^{N^1_j}(t)\rangle = \frac{\mathrm{d}}{\mathrm{d}t} \sum \limits_{i=1}^{N^1_j} \revision{m_h}\varphi(x_i^{N^1_j}(t), v_i^{N^1_j}(t))\\
= & \revision{m_h} \left[ \sum\limits_{i=1}^{N^1_j} \partial_x \varphi(x_i^{N^1_j}(t), v_i^{N^1_j}(t)) v_i^{N^1_j}(t) +\sum\limits_{i=1}^{N^1_j} \partial_v \varphi(x_i^{N^1_j}(t), v_i^{N^1_j}(t)) \right. \\
& \left. (H_1*_1(\mu^{N_j}+\nu_{N_j})+H_2*(\mu^{N_j} + \nu_{N_j}))(x_{i}^{N^1_j}(t), v_{i}^{N^1_j}(t))
\right].
\end{align*}
Thus for all $s \in [0, t_0^1)$, we have 
\begin{align*}
& \langle \varphi,\mu^{N_j^1}(s)- \mu^{N_j^1}(0)\rangle = \int_0^s \left[\int_{\mathbb{R} \times \mathbb{R}^{+}}\partial_x \varphi(x,v) v \right.\\ 
& \left. + \partial_v \varphi(x,v) (H_1*_1(\mu^{N_j} +\nu_{N_j})+H_2*(\mu^{N_j} + \nu_{N_j}))(x,v)\, \mathrm{d}\mu^{N_j^1}(t)(x,v)\right]\,\mathrm{d}t.
\end{align*}
Furthermore, 
\begin{equation}
\label{eqn_existence_1}
\lim\limits_{N_j^1 \to \infty} \langle \varphi, \mu^{N^1_j}(s) - \mu^{N^1_j}(0)\rangle = \langle \varphi, \mu_{*}^j - \mu_0^j\rangle. 
\end{equation}
By dominated convergence theorem, we obtain the limit (possibly for a sub-sequence) that 
\begin{equation} \label{eqn_existence_2}
\begin{aligned}
& \lim \limits_{N_j^1 \to \infty} \int_{0}^{s} \int_{\mathbb{R}\times \mathbb{R}^{+}} \left(\nabla_x \varphi(x,v) \cdot v\right) \,d\mu^{N_j^1}(t)(x,v)\,\mathrm{d}t =\\
=& \int_{0}^s\int_{\mathbb{R} \times \mathbb{R}^{+}} \left( \nabla_x \varphi(x,v) \cdot v \right) \, d_{\mu_{*}^j}(t)(x,v)\,\mathrm{d}t,
\end{aligned}
\end{equation}
for all $\varphi \in C_c^{\infty}(\mathbb{R} \times \mathbb{R}^{+})$. 
Furthermore, by Lemma \ref{lm_67} and Lemma \ref{lm_67_gwd}, we have, for every $\rho>0$, 
\begin{align*}
& \lim \limits_{N_j \to \infty} \left\| \left(H_1*_1(\mu^{N_j} + \nu_{N_j}) + H_2 * (\mu^{N_j} + \nu_{N_j})\right) -\right.\\
&\left.- \left((H_1*_1(\mu_{*}^j + \nu_{*}^j) + H_2 * (\mu_{*}^j + \nu_{*}^j)\right)\right\|_{L^{\infty}(B(0, \rho))} =0.
\end{align*}
Now since $\varphi \in C_c^{\infty} (\mathbb{R} \times \mathbb{R}^{+})$ has compact support, we obtain\\
\begin{center}
$
\lim \limits_{N_j \to \infty} \left\| \partial_{v} \varphi \left( \left(H_1*_1(\mu^{N_j} + \nu_{N_j}) + H_2 * (\mu^{N_j} + \nu_{N_j})\right) \right.\right.$\\
$\left. \left.- \left((H_1*_1(\mu_{*}^j + \nu_{*}^j) + H_2 * (\mu_{*}^j + \nu_{*}^j)\right) \right)\right\|_{\infty} =0$. 
\end{center}
Thus,
\begin{equation}
\label{eqn_existence_3}
\begin{aligned}
& \lim \limits_{k \to \infty} \int_{0}^s \int_{\mathbb{R} \times \mathbb{R}^{+}} \partial_v \varphi(x,v) (H_1*_1(\mu^{N_j} +\nu_{N_j})+\\
&+H_2*(\mu^{N_j} + \nu_{N_j}))(x,v)\, \mathrm{d}\mu^{N_j^1}(t)(x,v)\,\mathrm{d}t\\ 
= & \int_{0}^s \int_{\mathbb{R} \times \mathbb{R}^{+}} \partial_v \varphi(x,v) (H_1*_1(\mu^{N_j} +\nu_{N_j})+H_2*(\mu^{N_j} + \nu_{N_j}))(x,v)\, \mathrm{d}\mu^{N_j^1}(t)(x,v)\,\mathrm{d}t. 
\end{aligned}
\end{equation}
By the lane changing condition, we define 
\begin{align*}
\mu^{N_j^2}(t) =& \sum\limits_{i=1}^{N_{j-1}}\revision{m_h}\delta_{\left(x_i^{N_{j-1}}(t), v_i^{ N_{j-1}}(t)\right)}p \left([A^j + \Delta]_{+}, [A^{j} - A^{j-1} - \Delta]_{+}\right)\\
& -\sum\limits_{i=1}^{N_{j}}\revision{m_h}\delta_{\left(x_i^{N_{j}}(t), v_i^{ N_{j}}(t)\right)}p \left([A^j + \Delta]_{+}, [A^{j-1} - A^{j} - \Delta]_{+}\right)\\
& + \sum\limits_{i=1}^{N_{j+1}}\revision{m_h}\delta_{\left(x_i^{N_{j+1}}(t), v_i^{N_{j+1}}(t)\right)}p \left([A^j + \Delta]_{+}, [A^{j} - A^{j+1} - \Delta]_{+}\right)\\
& -\sum\limits_{i=1}^{N_{j}}\revision{m_h}\delta_{\left(x_i^{N_{j}}(t), v_i^{ N_{j}}(t)\right)}p \left([A^j + \Delta]_{+}, [A^{j+1} - A^{j} - \Delta]_{+}\right)
\end{align*}
where 
\begin{align*}
A^{j-1} & = \left(H_1*_1(\mu^{N_{j-1}}(t) + \nu_{N_{j-1}}(t))+H_2*(\mu^{N_{j-1}}(t) + \nu_{N_{j-1}}(t))\right)(x,v), \\
A^{j} & = \left(H_1*_1(\mu^{N_j}(t) + \nu_{N_j}(t))+H_2*(\mu^{N_j}(t) + \nu_{N_j}(t))\right)(x,v), \\
A^{j+1} & = \left(H_1*_1(\mu^{{N}_{j+1}}(t) + \nu_{N_{j+1}}(t))+H_2*(\mu^{{N}_{j+1}}(t) + \nu_{N_{j+1}}(t))\right)(x,v).
\end{align*}
Therefore, 
$\mu^{N_j}(t) = \mu^{N_j^1}(t) + \mu^{N_j^2}(t)$, and in addition, 
\begin{align*}
& \lim \limits_{N_{j-1} \to \infty}\sum\limits_{i=1}^{N_{j-1}}\revision{m_h}\delta_{\left(x_i^{ N_{j-1}}(t), v_i^{N_{j-1}}(t)\right)}p \left([A^j + \Delta]_{+}, [A^{j} - A^{j-1} - \Delta]_{+}\right)\\
=& \mu_{*}^{j-1} p\left(\left[\left(H_1*_1(\mu_{*}^{j}+\nu_{*}^{j})+H_2*(\mu_{*}^{j}+\nu_{*}^{j})\right)+\Delta\right]_+\right.,\\
&\left.\left[\left(H_1*_1(\mu_{*}^{j}+\nu_{*}^{j})+H_2*(\mu_{*}^{j}+\nu_{*}^{j})\right)-\right.\right.\\
&\left.\left.- \left(H_1*_1(\mu_{*}^{j-1}+\nu_{*}^{j-1})+H_2*(\mu_{*}^{j-1}+\nu_{*}^{j-1})\right)-\Delta \right]_{+}\right)\\
=& S^{j-1, j}(\mu_{*}^{j-1}, \mu_{*}^j).
\end{align*}
Furthermore, 
\begin{align}
\label{eqn_existence_4}
&\lim \limits_{N_j \to \infty} \mu^{N_j^2}(t) = \left(S^{j-1, j}(\mu_*^{j-1}, \mu_{*}^j)-S^{j, j-1}(\mu_{*}^{j-1}, \mu_{*}^j)\right)(1-\delta_{j,1})\\ \nonumber 
& +\left(S^{j+1, j}(\mu_{*}^j, \mu_{*}^{j+1}) - S^{j, j+1}(\mu^{j}, \mu^{j+1})(1-\delta_{j,m})\right)
= S(\mu_{*}^{j-1}, \mu_{*}^j, \mu_{*}^{j+1}).
\end{align}
The statement follows by combining equations \eqref{eqn_existence_1}, \eqref{eqn_existence_2}, \eqref{eqn_existence_3}, and \eqref{eqn_existence_4}. 
\end{proof}

\section{Conclusion} \label{sec_conclusions}

In this paper we have focused on a multi-lane multi-class description of vehicular traffic flow, where simultaneous presence of human-driven and autonomous vehicles has been considered.

The microscopic dynamics have been formulated by using a Bando-Follow-the-Leader type model, in which the interaction with the closest vehicle ahead is replaced by a space-dependent convolution kernel modeling interactions with the surrounding flow. Autonomous vehicles have been distinguished by control dynamics. Lane changing description has led to discrete events within the differential equations, and thus to a so-called hybrid system whose well-posedness has been studied. 

Inspired by the empirical fact that the {\em penetration rate} of the autonomous vehicles is nowadays small, we have computed a mean-field limit for the dynamics of the human-driven vehicles only, leading to a coupled system of a PDE and ODEs with discrete events. The discrete lane changing descriptions for human-driven vehicles has been modeled by a source term of the corresponding Vlasov-type equation. Existence and uniqueness study of the trajectories of this system has been performed. \revision{Moreover, the rigorous convergence of the finite dimensional hybrid system to the infinite dimensional hybrid system has been proved using the generalized Wasserstein distance.}

We point-out that the given application, based on traffic flow, inspiring this work is not restrictive, and many others may lead to the mathematical frameworks developed and studied here. More precisely, we refer to all physical multi-agent systems that are intrinsically characterized by heterogeneity and instantaneous jumps in one of their states. For instance, these include also models for air traffic control~\cite{Tomlin_1998}, chemical process control~\cite{Engell_2000} and manufacturing~\cite{Pepyne_2000}.

\section*{Acknowledgments}
G.~V. wishes to thank Benedetto Piccoli's Lab for the hospitality at Rutgers University and Michael Herty for supporting this research work.

\bibliographystyle{siamplain}
\bibliography{ref_traffic.bib}

\end{document}